\documentclass[12pt, reqno]{amsart}
\usepackage{amsaddr} 
\makeatletter
\renewcommand{\@setemails}{%
  \mbox{\itshape Email addresses:\space}%
  {\ttfamily\emails}.%
}
\makeatother
\usepackage{graphicx} 

\usepackage{thmtools} 

\usepackage[letterpaper,top=1in,            
  inner=1in,
  outer=1in,
  bottom=1in,
  headsep=4ex, ]{geometry}
\usepackage{parskip}

\usepackage[english]{babel} 
\usepackage[utf8]{inputenc} 
\usepackage[T1]{fontenc} 

\usepackage{amsmath} 
\usepackage{amssymb} 
\usepackage{amsfonts} 

\usepackage{mdframed} 
\mdfsetup{nobreak=true}

\usepackage{csquotes} 

\usepackage{float} 
\usepackage{tikz} 
\usetikzlibrary{cd} 

\usepackage{color} 
\usepackage{dsfont} 
\usepackage{mathrsfs} 
\usepackage{centernot} 
\usepackage{amsthm} 
\usepackage{mathtools} 
\usepackage{hyperref} 
\usepackage{mleftright} 
\mleftright

\usepackage[shortlabels]{enumitem} 
\setenumerate[0]{label=(\roman*)}

\usepackage{aligned-overset} 

\numberwithin{equation}{section} 

\newtheorem{theorem}{Theorem}[section]
\newtheorem*{theorem*}{Theorem}
\newtheorem{proposition}[theorem]{Proposition}
\newtheorem{lemma}[theorem]{Lemma}
\newtheorem{corollary}[theorem]{Corollary}

\theoremstyle{definition}
\newtheorem{definition}[theorem]{Definition}
\newtheorem{remark}[theorem]{Remark}
\newtheorem{example}[theorem]{Example}

\newtheorem{question}{Question}[section]
\newtheorem*{notation}{Notation}
\newtheorem{assumption}{Assumption}

\renewcommand{\subseteq}{\subset}
\renewcommand{\geq}{\geqslant}

\renewcommand{\leq}{\leqslant}

\usepackage{caption} 

\renewcommand{\epsilon}{\varepsilon}
\newcommand{\R}{\mathbb{R}}
\newcommand{\N}{\mathbb{N}}

\newcommand{\1}{\mathbf{1}}
\newcommand{\PP}{\mathbb P}
\newcommand{\QQ}{\mathbb Q}
\newcommand{\EE}{\mathbb E}

\newcommand*{\LL}{\mathop{}\!\mathcal{L}}
\newcommand*{\PPP}{\mathop{}\!\mathrm{P}}
\newcommand*{\QQQ}{\mathop{}\!\mathrm{Q}}
\newcommand{\proj}{\operatorname{proj}}

\newcommand{\Law}{\operatorname{Law}}

\newcommand{\W}{\operatorname{W}}

\newcommand{\HH}{\operatorname{H}}

\newcommand*{\diff}{\mathop{}\!\mathrm{d}}

\DeclarePairedDelimiter{\abs}{\lvert}{\rvert}
\DeclarePairedDelimiter{\norm}{\lVert}{\rVert}
\makeatletter
\let\oldabs\abs
\def\abs{\@ifstar{\oldabs}{\oldabs*}}
\makeatother
\makeatletter
\let\oldnorm\norm
\def\norm{\@ifstar{\oldnorm}{\oldnorm*}}

\newcommand{\cal}{\mathcal}

\usepackage{tikz}
\usetikzlibrary{calc}
\usetikzlibrary{intersections}
\usetikzlibrary{positioning}
\usetikzlibrary{hobby}

\usetikzlibrary{shadings}

\usetikzlibrary{decorations}

\usepackage{pgfplots}
\usetikzlibrary{intersections, pgfplots.fillbetween}
\pgfplotsset{compat=1.18}

\date{\today}

\title{Quantitative QSD convergence in 1-Wasserstein distance via the Föllmer drift}

\author{Joaquín Fontbona\textsuperscript{\ensuremath{\dagger}} \and Pablo López-Rivera\textsuperscript{\ensuremath{\ddagger}}}

\address{\textsuperscript{\ensuremath{\dagger}}Centro de Modelamiento Matemático, Universidad de Chile \& IRL 2807 - CNRS\\
\textsuperscript{\ensuremath{\ddagger}}Department of Mathematics, UCLA}
\email{\textsuperscript{\ensuremath{\dagger}}fontbona@dim.uchile.cl}

\email{\textsuperscript{\ensuremath{\ddagger}}plopez@math.ucla.edu}

\calclayout

\begin{document}

\begin{abstract}
We develop a novel pathwise approach to study the convergence of the law of killed diffusion processes conditioned on non-absorption, towards a quasi-stationary distribution (QSD) as time goes to infinity. We start from the general observation that the dynamics of an absorbed Markov process conditioned upon survival up to time $T>0$ is the minimizer of the pathwise relative entropy with respect to its unconditioned dynamics, under a simple distributional constraint at that time; in other words, a F\"ollmer process. We then show how this result applies to a Brownian diffusion process softly-killed at a state-dependent regular rate, and characterize the  minimizer and its drift change. In the case when the diffusion process is moreover reversible, we leverage these ideas and recent results on the propagation of weak log-concavity of HJB semigroups to prove that, under strict asymptotic convexity of the potential, the conditioned dynamics satisfies a contractivity property in $1$-Wasserstein distance, uniformly in $T>0$. We then establish the existence of a QSD with a large domain of attraction and deduce the exponentially fast convergence to it of the conditioned semigroup in the $1$-Wasserstein distance as $T$ goes to infinity. Finally, we deduce the exponentially fast convergence, also in $1$-Wasserstein distance, of the law of the corresponding Q-process towards its equilibrium. 
\end{abstract}

\maketitle

\setlength{\parskip}{0.5\baselineskip}
\setlength{\parindent}{20pt}

\section{Introduction and main results}
\label{sec:introduction}

Consider a càdlàg Markov process $X = (X_t)_{t \geq 0}$ in $\R^d$ defined on some probability space $(\Omega, {\cal F}, \PP)$ with $\partial \subset \R^d$ an absorbing Borel set for $ X$; that is, such that
$$\forall t \geq0, \quad {X}_t= {X}_{t\wedge\tau_{\partial}} ,  \,  \PP_{\nu}\mbox{-a.s.}$$
Here and in all the sequel, $\PP_{\nu}$ denotes  the law of $ X$ when initialized with a distribution $\nu$ on $\R^d$, and $\tau_{\partial}$ stands for the stopping time (with respect to the natural filtration of $X$) given by
\begin{equation*}
\label{eq:taudelta}
\tau_{\partial}\coloneq \inf \{t\geq 0: {X}_t\in \partial\}.
\end{equation*}
The long-time behavior of the law of  $ X_t$ conditioned on the event $\{\tau_{\partial}>t\}$ of not yet being absorbed by time $t\geq 0$ has historically been an active research domain in the theory of stochastic processes. Central  to this problem is the notion of {\it quasi-stationary distribution} (QSD), namely, a probability law $\alpha$ supported on $\R^d \setminus \partial$ such that 
\begin{equation}\label{eq:QSDdef}
\forall t \geq 0, \quad \alpha = \PP_{\alpha}( X_t\in \cdot | \tau_{\partial}>t). 
\end{equation}
Conditions for the existence and uniqueness of a QSD $\alpha$ and, moreover, for $\alpha$ to also arise as the suitable distributional limit of $X_t$ conditioned on not being absorbed after an arbitrarily long time, that is,
 \begin{equation}\label{eq:QSDconvergence}
   \PP_{\nu}( X_t\in \cdot | \tau_{\partial}>t) \xrightarrow[t \to +\infty]{}  \alpha,
\end{equation}
have been extensively studied in the literature too. We refer the reader to \cite{MR2994898, MR2986807} for general background and historical references; see also Section \ref{subsec:literatureQSD} for a bibliographical discussion on quasi-stationarity.

In the present work, we are interested in the pathwise behavior  of the   process $X$ conditioned on non-absorption on finite time intervals, and in deriving quantitative rates for the convergence \eqref{eq:QSDconvergence} from  properties of the conditioned dynamics. 

\subsection{Main results}

A key observation underlying our approach is a yet unexploited (and apparently unnoticed) variational interpretation of the conditioned process, to be established in Theorem \ref{thm:main_i}, namely that the pathwise law of the process conditioned on non-absorption is the minimizer of the relative entropy with respect to the original unconditioned law of $X$, under a simple constraint on the law of a certain observable at time $T$. This brings forward the possibility of studying Markov processes conditioned on non-absorption as particular instances of \textit{F\"ollmer processes}, which is the main motivation of the present work; see Section \ref{subsubsec:the_gaussian_follmer_process} for this notion and related references.

From this starting point, our aim is to establish a general methodology and specific tools to develop a pathwise approach to obtain quantitative QSD convergence results. To that end, we focus here on softly-killed Brownian diffusion processes conditioned on survival, showing first that they enter the general framework of absorbed Markov processes covered by Theorem \ref{thm:main_i}. We  give in Theorem \ref{thm:main_ii_0} a pathwise description of the conditioned law in this case and, in particular, of the \textit{F\"ollmer drift} providing the optimal correction on the original diffusion, needed to achieve the distributional constraint associated with the conditioning. This result also provides an interpretation of the optimal pathwise entropy in terms of the solution to a Hamilton-Jacobi-Bellman (HJB) equation. It is worth noting that this variational (F\"ollmer process) interpretation of the conditioned process requires the soft-killing to be suitably represented as part of the sample path dynamics, and it does not hold at the sole level of the corresponding penalized (Feynman-Kac) diffusion semigroup.

The central goal of our program is to exploit the explicit dynamics of the conditioned process given by Theorem \ref{thm:main_ii_0} in order to provide sufficient conditions for the convergence \eqref{eq:QSDconvergence} to take place at exponential rate in the transport $1$-Wasserstein distance, which we denote by $\W_1$, for a large class of initial distributions $\nu$. This is achieved in the reversible case in our main result,  Theorem \ref{theo:domain_of_attraction_W1},  stated next at an informal level.

\begin{theorem*}[Informal version of Theorem \ref{theo:domain_of_attraction_W1}]
Let $B = (B_t)_{t \geq 0}$ be a standard Brownian motion on $\R^d$ and let $U \colon \R^d \to \R$ be a regular potential (see Assumption \ref{ass:langevin}). Let $X = (X_t)_{t \geq 0}$ be the solution to the  SDE
\begin{equation}
\label{eq:langevin_intro}
\diff X_t = \diff B_t - \nabla U(X_t) \diff t,
\end{equation}
and let $\diff \mu \propto e^{-2U}\diff x$ be its reversible distribution. Suppose further that $X$ is softly killed at a regular position-dependent rate $\rho \colon \R^d \to \R_+$. Then, under weak asymptotic convexity conditions involving  both $U$ and $\rho$ (see Assumption \ref{ass:effective_potential}), there exist constants $C^*, c^* > 0$ such that
for  all $\eta, \nu$  initial distributions with finite first moment, 
\begin{equation}
\label{eq:contractivity_intro}
 \W_1(\nu_t^T, \eta_t^T) \leq C^* e^{-c^*t}\W_1(\nu_0^T, \eta_0^T),   \quad \forall T > 0, \forall t \in [0,T], 
\end{equation}
where $\nu_t^T$ and $\eta_t^T$ denote the laws at time $t$ of the process \eqref{eq:langevin_intro} conditioned on surviving up to time $T$, when respectively initialized with laws  $\nu$ and $\eta$. 
Moreover, there exists a QSD $\alpha$ for the conditioned process and,  for every $\nu $ equal to a delta mass or such that $\nu \ll \mu$ with $\frac{\diff \nu}{\diff \mu} \in L^q(\mu)$ for some $q > 1$, there exists $K(\nu) > 0$ such that
\[\W_1(\nu_T^T,\alpha) \leq K(\nu)e^{-c^*T} \quad\forall T > 0 .\]
Last,  $\alpha$ is the unique QSD such that  $\alpha \ll \mu$ and $\frac{\diff \alpha}{\diff \mu} \in L^q(\mu)$ for some $q > 1$. 
\end{theorem*}

The proof of this result has two intermediate stages. First,  leveraging the general parallel between pathwise entropy minimization and certain stochastic optimal control problems, together with recent results on the propagation of weak log-concavity of HJB semigroups \cite{chaintron25propagation}  we prove, under conditions of strict asymptotic convexity of the potential \cite{MR2843007, eberle16reflection}, that the stochastic flow associated with the conditioned dynamics in each interval $[0,T]$ satisfies a uniform in $T>0$ contractivity estimate in the $1$-Wasserstein distance. This is the content of Theorem \ref{thm:contraction_dynamics_intro} and Corollary \ref{coro:contractionconseq}, and directly yields the bound \eqref{eq:contractivity_intro} when combined with Theorem \ref{thm:main_ii_0}.

The second stage of the proof of Theorem \ref{theo:domain_of_attraction_W1} consists in identifying conditions for a QSD $\alpha$ to exist,  and moreover a class of initial conditions $\nu$ and $\eta$ for which the $1$-Wasserstein distance between their ``tilted'' versions $\nu_0^T$ and $ \eta_0^T$ as in \eqref{eq:contractivity_intro} can be controlled uniformly in $T>0$. We introduce to that end the notion of weak $p$-ergodicity with $p \in [2,+\infty)$ (see Definition \ref{def:fk_ergodicity}), as a general condition for Feynman-Kac semigroups that systematizes several classic and more recent frameworks in which quasi-stationarity has been previously studied. We show in Theorem 
\ref{theo:domain_of_attraction} that this property yields the existence of a QSD and convergence to it in a suitable domain of attraction, which includes initial distributions with sufficiently integrable densities relative to the reversible measure.  Moreover, we show that under the assumptions of Theorem \ref{theo:domain_of_attraction_W1} this property holds for every $p \in [2,+\infty)$ and provides suitable controls as $T\to +\infty$ of integrals of the tilted measures $\nu_0^T$ and $ \eta_0^T$, which include the required moments.

Finally, we focus on the so-called \textit{Q-process}, widely studied in the framework of quasi-stationarity, which can be understood as the process $X$ conditioned on never being killed. Again in the reversible case and under strict asymptotic convexity assumptions, we prove in Theorem \ref{prop:q_process_exp_fast_intro} the exponentially fast convergence, also in the $1$-Wasserstein distance, of the law of the Q-process towards its equilibrium.

We end this section with an outline of the structure of the paper and with some basic notation to be used in the sequel; additional notation will be set when required.

\subsection{Organization of the paper}
\label{subsec:organization}

In Section \ref{sec:related_literature}, we review the literature relevant to our work and discuss some connections between them. In Section \ref{sec:variational}, we state the variational characterization of conditioned processes in Theorem \ref{thm:main_i}. In Section \ref{subsec:condiffkillfoll}, in the case of softly-killed diffusion processes, we characterize the associated F\"ollmer drift in Theorem \ref{thm:main_ii_0} and state the entropy representation formula for it in Corollary \ref{cor:valueentropy}. In Section \ref{sec:cond_contractivity}, we state Theorem \ref{thm:contraction_dynamics_intro} and Corollary \ref{coro:contractionconseq} on the  Wasserstein contractivity properties of the conditioned dynamics, in terms of asymptotic convexity assumptions, which are also detailed in that section. In Section \ref{sec:qsd}, we first introduce the notion of weak $p$-ergodicity, under which we establish in Theorem \ref{theo:domain_of_attraction} the existence of a QSD and sufficient conditions for a distribution to belong to its domain of attraction. We also prove that under our assumptions, the weak $p$-ergodicity holds for every $p \in [2,+\infty)$ in Proposition \ref{prop:4+5weakergo}. We end Section \ref{sec:qsd} by stating Theorem \ref{theo:domain_of_attraction_W1} on the Wasserstein convergence of the conditioned semigroup. In Section \ref{sec:q_process}, we introduce the Q-process and state Theorem \ref{prop:q_process_exp_fast_intro}. All the proofs are provided in Section \ref{sec:proofs}, along with the additional technical background that is required. 

\begin{notation}
$~$
\begin{itemize}
    \item We write $\N \coloneq \{0,1,2,\dots\}$ and $\N^*\coloneq \N \setminus \{0\}$. $\R$ denotes the set of real numbers. Both $\R_+$ and $\R_{\geq 0}$ denote the set of nonnegative real numbers, and we denote by $\R_{>0}$ the set of positive real numbers.

    \item We write $ X_{[0,T]} \coloneq ({X}_t)_{t \in [0, T]}$ and denote by $\PP^T_\nu $ the law of $X_{[0,T]}$ on the Skorokhod space $\mathbb{D}_T \coloneq \mathbb{D}([0,T];\R^d)$.
The associated expectation is written  $\EE^T_\nu$. When $\nu=\delta_x$, we use $x$ as a subscript instead. 

    \item If $E$ is a Polish space, we denote by ${\cal P}(E)$ the set of Borel probability measures on $E$, by $\mathbb{B}(E)$ the class of real-valued bounded and measurable functions on $E$, and by $C_\mathrm{b}(E)$ the class of real-valued bounded and continuous functions on $E$. For $E=\R^d$,  ${\cal P}_1(\R^d)$ denotes the subset of ${\cal P}(\R^d)$ of probability laws with finite first moment. 

   \item We denote by $\Law(Z)$ the law of a random variable or process $Z$. Moreover, if $\Law(Z) = \mathbb{Q}$ and $Y$ is a random variable of the form  $Y=\Phi(Z)$ for some measurable function $\Phi$, the law of $Y$ is denoted $\Law_{\mathbb{Q}}(Y)$.

   \item For $\mu \in \mathcal{P}(\R^d)$ and $\phi \colon \R^d \to \R$, we write $\langle \mu, \phi\rangle \coloneq \int_{\R^d} \phi \diff \mu$. For $\psi, \phi \in L^2(\mu)$, their inner product in the Hilbert space $L^2(\mu)$ is denoted by $\langle \psi, \phi\rangle_{\mu} \coloneq \int_{\R^d} \psi \phi \diff \mu$.
\end{itemize}
\end{notation}

\subsection*{AI usage declaration}

The initial version of this work posted on the arXiv was generated solely by the authors. In this new version, a large language model (LLM)  was used for a final check of mathematical correctness and language. The LLM suggested a way of removing an unnecessary assumption in Theorem \ref{theo:domain_of_attraction_W1}, already implied by previous ones, and fixed minor mistakes, all of which was taken into account but written in the authors' own words. The LLM also suggested oversimplifications that were not correct, which resulted in clarifications added by the authors to prevent possible similar misunderstandings by the final readers.  
We omit the LLM brand used, as we refrain from advertising for particular private companies.

\subsection*{Acknowledgments}
\label{subsec:acknowledgments}
We would like to thank Nicolas Champagnat and Giovanni Conforti for their interest and for motivating discussions. JF acknowledges partial support from ANID Fondecyt Project 1242001 and  ANID-BASAL CMM FB21000.

\section{Connection to the literature}
\label{sec:related_literature}

In this section we discuss the literature to which our work relates, focusing on quasi-stationarity, F\"ollmer processes, and weak asymptotic convexity. We also highlight the contributions of our work in those contexts. 

\subsection{Quasi-stationarity}
\label{subsec:literatureQSD}

Quasi-stationarity is as of today a broad subject, so we focus here on the diffusive case; see e.g. the survey article \cite{MR2994898} or the book \cite{MR2986807} for wider discussions and further references. The study of quasi-stationarity in the diffusive setting was pioneered in the one-dimensional case in \cite{MR137143} and in the multi-dimensional setting in \cite{MR781410, MR970476}; see also \cite{MR1303922, MR2262851, MR2917771, MR2561437, MR1642845, MR1781008, MR2606515, MR3012095, MR3833098, MR4014294}, to name a few, for further developments and applications, including in mathematical ecology and population models. The study of the convergence to the Q-process in the diffusive case can be traced back to \cite{MR781410, MR970476, MR1349173}.

On the quantitative side, the works e.g. \cite{MR2986807, MR4077228, MR3449390, MR3663104, zbMATH07450813, zbMATH07730868, zbMATH07515386, MR4546021, MR4756950, guillin2026longtimebehaviorkilled, MR5012340, villemonais2026quasicompactnessdominatedkernelsapplication} provide general criteria yielding convergence to a QSD in the total variation distance. In the setting of hard-killed diffusions, exponentially-fast convergence is obtained from a Poincaré inequality and the Bakry-Émery criterion in \cite{MR4278594}, with $1$-Wasserstein bounds also derived in specific cases of additive potentials. Another approach based on functional inequalities was applied to obtain quantitative rates in Wasserstein distance for systems of coupled quantum harmonic oscillators in \cite{MR4627339}. In \cite{MR4095047}, a polynomially fast rate of convergence in $1$-Wasserstein distance is obtained for the marginal law of a Brownian motion with drift, conditioned not to have reached $0$.

The work \cite{MR4940326} provides a general criterion for the exponential convergence, with respect to a Wasserstein-type distance induced by a bounded metric, to the QSD of softly-killed processes with a killing rate of finite oscillation, based on the existence of an exponentially penalized contractive coupling (see Assumption (A) in that work). This criterion applies in particular to Bernoulli convolutions and iterated random contracting functions, and is intimately linked to the validity of a global Harnack-type inequality (Assumption (H) in \cite{MR4940326}, shown to hold in Proposition 2.2 therein under their Assumption (A)). We notice that this type of inequality is hard to satisfy uniformly on unbounded (with respect to the Euclidean metric) domains in the diffusive case. Its proof strongly relies on their contractivity condition (A), which is stated in terms of a bounded underlying metric. See, however, the discussion at the end of Subsection  \ref{subsec:lit_weak_cvx} below.

Theorem \ref{theo:domain_of_attraction_W1} in this work provides, to the best of our knowledge, a novel criterion for the exponential convergence in a Wasserstein-type distance induced by an unbounded metric in the case of reversible diffusions. We also remark that the assumptions on the drift and the killing rate in Theorem \ref{theo:domain_of_attraction_W1} concern asymptotic bounds on weak convexity profiles, hence going beyond strong convexity conditions as implicitly assumed in the work \cite{MR4278594}.

\subsection{F\"ollmer processes}
\label{subsubsec:the_gaussian_follmer_process}

With the aim of providing general theoretical context for the leading ideas in this work, we recall here the classic notions of F\"ollmer process and F\"ollmer drift,  named so after seminal works by Hans Föllmer \cite{MR798318, MR838561, MR983373}. 

Only for the purpose of this subsection, denote by $(\Omega, \mathcal{F}, (\mathcal{F}_t)_{t \in [0,1]}, \PP)$ the filtered Wiener space in $\R^d$,  so that the canonical process $X = (X_t)_{t \in [0,1]}$ is a standard $d$-dimensional Brownian motion under  $\PP$. In particular, $\Law_\PP(X_1) = \gamma_d$, where $\gamma_d$ denotes the $d$-dimensional standard Gaussian measure. Let $\EE$ denote the expectation with respect to $\PP$ and let $(\PPP_t)_{t \in [0,1]}$ be the (Brownian) heat semigroup. Consider also $\eta \ll \gamma_d$ with $g \coloneq \frac{\diff \eta}{\diff \gamma_d}$ and define the law $\widetilde \PP \in \mathcal{P}(\Omega)$ of the process $X$ \textit{conditioned to have terminal distribution $\eta$}, as follows: for $F\colon \Omega \to \R$ bounded and measurable,
\begin{equation}
\label{eq:law_gaussian_follmer}
\widetilde \EE\left[F(X)\right] \coloneq \EE[F(X)g(X_1)]. 
\end{equation}
Clearly, $\Law_{\widetilde\PP}(X_1) = \eta$. The measure $\widetilde{\PP}$ verifies the following remarkable properties.

\begin{theorem}
\label{thm:follmer}
Let $\eta \in \mathcal{P}(\R^d)$ with $\eta \ll \gamma_d$ and $g \coloneq \frac{\diff \eta}{\diff \gamma_d}$, and define $\widetilde{\PP}$ by \eqref{eq:law_gaussian_follmer}.
\begin{enumerate}
    \item \label{item:i:thm:follmer} If $\HH(\cdot|\PP)\colon \mathcal{P}(\Omega) \to \R_+ \cup \{+\infty\}$ denotes the relative entropy functional with respect to $\PP$, then
    \begin{equation}
    \label{eq:gaussian_follmer_minimal_entropy}
    \HH(\widetilde\PP|\PP) = \inf_{\substack{\QQ \in \mathcal{P}(\Omega) \\ \Law_{\QQ}(X_1) = \eta}} \HH(\QQ| \PP).
    \end{equation}

    \item \label{item:ii:thm:follmer} Consider the process $\widetilde{X} = (\widetilde X_t)_{t \in [0,1]}$ given by
    \begin{equation}
    \label{eq:gaussian_follmer_correction}
    \forall t \in [0,1], \quad\widetilde X_t \coloneq X_t - \int_0^t \nabla \ln \PPP_{1-s}g(X_s) \diff s.
    \end{equation}
    Then $\widetilde{X}$ is a $\widetilde{\PP}$-Brownian motion. In particular,
    \begin{equation}
    \label{eq:gaussian_follmer_dynamics}
    \widetilde\PP\text{-a.s. } \forall t \in [0,1], \quad \diff X_t = \diff \widetilde{X}_t + \nabla \ln \PPP_{1-t}g(X_t) \diff t, \quad X_0 = 0.
    \end{equation}

    \item \label{item:iii:thm:follmer} The relative entropy of $\widetilde \PP$ with respect to $\PP$ admits the following representation:
    \begin{equation}
\label{eq:gaussian_follmer_entropy}
    \HH(\widetilde \PP|\PP) = \widetilde\EE\left[\int_0^1\frac12 \abs{\nabla \ln \PPP_{1-t}g(X_t)}^2\diff t\right].
    \end{equation}
\end{enumerate}
\end{theorem}

That is, part \ref{item:i:thm:follmer}  states that the law $\widetilde{\PP}$ has minimal entropy among all the path laws with terminal distribution $\eta$;  \ref{item:ii:thm:follmer} provides the drift correction under the measure change \eqref{eq:gaussian_follmer_correction} so that the canonical process solves the Föllmer stochastic differential equation \eqref{eq:gaussian_follmer_dynamics}; and finally, \ref{item:iii:thm:follmer} provides a formula for the value of the relative entropy $\HH(\widetilde \PP|\PP)$ in terms of the \textit{Föllmer drift} \eqref{eq:gaussian_follmer_entropy}. Bringing all three statements together reveals a key insight: if the $\inf$ in \eqref{eq:gaussian_follmer_minimal_entropy} is finite, the \textit{F\"ollmer process} defined by \eqref{eq:gaussian_follmer_dynamics} is the (only) one that realizes the distributional constraint $\Law_{\widetilde\PP}(X_1) = \eta$ at minimum possible ``energy'' or cost, given by the right-hand side of \eqref{eq:gaussian_follmer_entropy}.

Despite its interest from several research areas including stochastic analysis \cite{MR1095665, MR1262893, MR1919610, leonard11stochastic}, geometric and functional inequalities \cite{MR3112438, MR3782065, MR3485296, MR3904394, MR4116725, MR4126934, MR4343907, MR4663509, MR4282081, MR4797372, MR4864872, MR4989825, MR5103350}, optimal transport \cite{MR2873864, MR3121631, MR3489825}, and statistics \cite{pmlr-v99-tzen19a, bortoli2021diffusion, pmlr-v235-chen24n}, the relation between F\"ollmer processes and Markov processes conditioned on non-absorption remains, to our knowledge, largely unexplored.

Theorems \ref{thm:main_i}, \ref{thm:main_ii_0}, and Corollary \ref{cor:valueentropy} below together provide an analog of Theorem \ref{thm:follmer} for the path law of a softly-killed Brownian diffusion process conditioned on survival; see also the discussions in Sections \ref{sec:variational} and \ref{subsec:condiffkillfoll}.

\subsection{Weak asymptotic convexity}
\label{subsec:lit_weak_cvx}

The notions of weak convexity profile and strict asymptotic convexity of a potential,  presented in Sections \ref{sec:cond_contractivity} and \ref{sec:weak_convexity_Eberle}, are key elements of our main result. These notions can be traced back to the works of Eberle \cite{MR2843007, eberle16reflection} (see also Theorem \ref{thm:eberle} below), which unveiled their connection to certain contractivity properties in Wasserstein-type distances for diffusion processes, obtained via reflection couplings (see \cite{lindvall86coupling, MR1924231}).  We refer the reader to \cite{MR3330820, MR3980913, MR3939573, MR4133372, zbMATH08244987} for other applications. 

Of particular relevance for the present work is the use of these ideas in the context of stochastic control and the propagation of strict log-concavity of HJB semigroups \cite{MR4674060, MR4771110, ConfortiLackerPal2025, MR5022995, conforti25km, chaintron25propagation};  Proposition \ref{prop:conforti} below summarizes the main results from \cite{chaintron25propagation} we will rely on. We notice that, although  those results are obtained on the basis of the stochastic optimal control interpretation of the HJB equation, it is also possible to state them directly in terms of the entropy viewpoint presented here. In fact, Theorem \ref{thm:contraction_dynamics_intro} about the contractivity of the conditioned dynamics ultimately is a consequence of the asymptotic behavior of the relative pathwise entropy studied in Theorem \ref{thm:main_i}, as a function of the initial position of the process (see Remark \ref{rem:asymptotEntropy}). 

Finally,  let us mention that we believe  that reflection coupling techniques as in \cite{conforti25km,chaintron25propagation} could in principle be used to construct penalized contractive couplings satisfying conditions analogous to Assumption (A) in \cite{MR4940326} with respect to various (but not necessarily bounded) distances. The in-depth development of this connection is left for future work.

\section{Variational characterization of the conditioned dynamics}
\label{sec:variational}

The following basic assumption provides the general framework for our first result.

\begin{assumption}
\label{ass:a1}
$\partial \subset \R^d$ is a Borel set, $X= ({X}_t)_{t \geq 0}$ is a càdlàg Markov process  in $\R^d$ such that ${X}_t={X}_{t\wedge \tau_{\partial}}$ for all $t\geq 0$, and its initial distribution $\nu \in {\cal P}(\R^d)$ satisfies $\nu(\R^d\setminus \partial) = 1.$
\end{assumption}

Let $T > 0$, and assume that $\PP_\nu(\tau_\partial>T) > 0$. The \textit{law of $X$ conditioned on non-absorption up to time $T$} is the probability measure $\widetilde{\PP}^T_\nu\ll \PP^T_\nu$ on $\mathbb{D}_T$ defined by the relation 
\begin{equation}
\label{eq:law_hat_x}
\forall F \in \mathbb{B}(\mathbb{D}_T), \quad \widetilde{\EE}^T_\nu\left[F({X}_{[0,T]})\right] \coloneq \frac{1}{\PP^T_\nu(\tau_{\partial}>T)}\EE^T_\nu\left[F({X}_{[0,T]}) \mathbf{1}_{\{\tau_{\partial}>T\}}\right].
\end{equation}

We start by providing a variational characterization  of the measure $\widetilde{\PP}^T_\nu$, namely as the minimizer of the pathwise entropy relative to  the unconditioned law $\PP^T_\nu$, among all probability laws $\QQ$ on  $\mathbb{D}_T$ satisfying a simple distributional
constraint  at time $T$.  
More precisely,  let us introduce the $\{0,1\}$-valued 
process $N=(N_t)_{t\geq 0}$ given by 
\begin{equation}
\label{eq:NtindXt}
\forall t \geq 0, \quad N_t \coloneq \mathbf{1}_{\{{X}_t  \in \partial\}},
\end{equation} 
which by Assumption \ref{ass:a1} satisfies ${N}_t= {N}_{t\wedge\tau_{\partial}}$ and $\{\tau_{\partial} >t\}= \{ N_t = 0 \}$ for all $t \geq 0$ $\PP_{\nu}$-a.s. Note that $$\Law_{\widetilde{\mathbb{P}}^T_{\nu}}( N_T )=\delta_0.$$ Our first result reads as follows.

\begin{theorem}
\label{thm:main_i}
Fix $T > 0$, assume that $\PP_\nu(\tau_\partial>T) > 0$, and let Assumption \ref{ass:a1} hold. Let $\varphi \colon \R \to \R\cup \{+\infty\}$ be a convex function and define the functional $\mathbb{H}^T_\varphi \colon \mathcal{P} (\mathbb{D}_T)\to \R\cup \{+\infty\}$ by
 \[\forall \QQ \in \mathcal{P} (\mathbb{D}_T), \quad  \mathbb{H}^T_\varphi (\mathbb{Q}) \coloneq \begin{cases} \EE_{\nu}^T \left[\varphi\left(\dfrac{\diff \mathbb{Q}}{\diff  \PP^T_{\nu}}\right)\right], & \text{if } \mathbb{Q}\ll \PP^T_{\nu} \\ +\infty, & \text{otherwise}.\end{cases}\]
Then we have
\begin{equation}\label{eq:variational}\mathbb{H}^T_\varphi (\widetilde{\PP}^T_{\nu}) = \min_{\substack{\QQ \in \mathcal{P}(\mathbb{D}_T) \\ \Law_{\mathbb{Q}}(N_T)= \delta_0}}\mathbb{H}^T_\varphi (\mathbb{Q})
\end{equation}
and the minimizer is unique if moreover $\varphi$ is strictly convex and the minimum is finite. In particular, if we let $\varphi(r) = r \ln r \1_{\{r \geq 0\}} + \infty \1_{\{r < 0\}}$, we see that  $\mathbb{Q}=\widetilde{\PP}^T_{\nu}$ is the unique minimizer  of the  relative entropy $\mathbb{H}^T_\varphi(\mathbb{Q}) \eqcolon \HH(\mathbb{Q} | \PP^T_\nu)$ over the set $\{\mathbb{Q} \in {\cal P} ( \mathbb{D}_T): \Law_{\mathbb{Q}}(N_T)= \delta_0\}$.
\end{theorem}

The previous variational characterization connects the conditioned process with the framework of F\"ollmer processes discussed in Section \ref{subsubsec:the_gaussian_follmer_process}. Theorem \ref{thm:main_i} is in part inspired by extensions of those ideas developed in \cite{MR1919610}. It also suggests a possible pathwise viewpoint to study the behavior of the conditioned dynamics as $T\to +\infty$, and in particular, the QSD convergence  \eqref{eq:QSDconvergence}. The proof of Theorem \ref{thm:main_i} is deferred to Section \ref{subsec:proof_main_i}.
 
 In what follows, our aim is to develop in detail this pathwise approach and apply it to study the dynamics of softly-killed Brownian diffusion processes conditioned on survival.

\section{Conditioned  softly-killed diffusions  as F\"ollmer processes}\label{subsec:condiffkillfoll}

Consider a locally bounded nonnegative Borel function $\rho \colon \R^d \to \R_+$. The Markov process $X$ \textit{softly killed at state-dependent rate $\rho$} is classically  defined as the process equal to  $X$ on the interval $[0,\tau)$, and equal to a cemetery state $\dagger  \notin \mathbb{R}^d$ on the interval $[\tau, +\infty)$, where the  \textit{killing time} $\tau$ is a random variable taking values in $\R_+ \cup \{+\infty\}$ such that 
\begin{equation}
\label{eq:distribution_tau_exponential1}
\forall T>0, \quad \PP(\tau > T | X_{[0,T]})  = \exp\left(-\int_0^T \rho(X_s)\diff s\right), 
\end{equation}
 defined in a suitable enlargement of the probability space $(\Omega, {\cal F}, \PP)$ (here equally denoted for simplicity). A classic way of doing this is introducing an exponential random variable  $\mathcal{E}\sim \mathrm{exp}(1)$ independent of $X$, and setting $\tau \coloneq \inf\left\{t > 0 : \int_0^t \rho(X_s)\diff s > \mathcal{E} \right\}$; see Section \ref{subsec:prelim} for an alternative construction that will be more useful for our technical approach.

On the so-called \textit{event of survival} $\{\tau>T\}$, this killed process can be seen as a particular case of the previously considered càdlàg absorbed process, conditioned on non-absorption. More precisely, $X$ is the first coordinate of the absorbed Markov process $\bar{X}$ in $\R^{\bar{d}} \coloneq \R^{d+1}$ given by 
\begin{equation}\label{eq:absorbedXN}\forall t \geq 0, \quad \bar{X}_t\coloneq (X_{t\wedge \tau},N_{t}) \mbox{ with }N_t \coloneq \mathbf{1}_{\{\tau\leq  t\}},
\end{equation}
which has the absorbing set $\bar{\partial }=\R^d\times \N^*$ and the absorbing time $\tau_{\bar{\partial}}=\tau$, under initial distributions of the form $\bar{\nu} \coloneq \nu \otimes  \delta_0\in {\cal P}(\R^{\bar{d}})$ satisfying Assumption \ref{ass:a1} in this setting.

For notational simplicity,   in this framework we identify the probability law  $ \bar{\nu} $ with $\nu$. Moreover we 
write
$\PP^T_{\nu}  $ instead of ${\PP}^T_{\bar{\nu}  }$ for the law of  $\bar{X}_{[0,T]}$,  and $\widetilde{\PP}^T_{\nu}$    instead of $  \widetilde{{\PP}}^T_{\bar{\nu} },$
for its law  when conditioned on the event $\{\tau_{\bar{\partial}}>T\} =\{N_T=0\} $.  We can thus write in this case
$$\frac{\diff  \widetilde{\PP}_{\nu}^T}{\diff  \PP^T_{\nu}}=\frac{\mathbf{1}_{\{N_T =0 \}}}{\PP^T_{\nu}( N_T = 0) }.$$

In the next results we will make use of the following additional assumptions and notation.

\begin{assumption}
\label{ass:a2}
$(X_t)_{t\geq 0}$ is a nonexplosive strong solution in $\R^d$ of the stochastic differential equation
\begin{equation}
\label{eq:dynamics_x}
    \diff X_t = \diff B_t + b(X_t)\diff t, \quad X_0 \sim \nu
\end{equation}
with $B = (B_t)_{t \geq 0}$ a standard Brownian motion in $\R^d$ and $X_0$ independent from it. 
Moreover,  the vector field $b\colon \R^d \to \R^d$ is locally Lipschitz so that pathwise uniqueness for \eqref{eq:dynamics_x} holds. 
\end{assumption}

\begin{assumption}
\label{ass:uniqueFK} The function $\rho \colon \R^d \to \R_+$ is continuous and nonnegative, and, for each $T>0$, the {\it backward} Cauchy problem
\begin{equation}
\label{eq:pde_fT}
\begin{cases}
(\partial_t \phi  + \LL^X \phi )(s,y) = \rho (y) \phi(s,y), \quad (s,y)\in  [0,T)\times  \mathbb{R}^d\\
\phi(T,y) = 1,  \quad y\in \mathbb{R}^d , 
\end{cases}
\end{equation}
with the operator $\LL^X$ defined below, has a bounded classical solution with polynomially growing derivatives. 
\end{assumption}

\begin{notation}
$~$
\begin{itemize}
\item We denote by $\LL^X$ the infinitesimal generator of the  process $X$ (without killing). 

\item The c\`adl\`ag, complete filtration  generated by $B$ and $X_0$ in  (the enlarged space) $(\Omega, \mathcal{F}, \mathbb{P})$   is denoted   $(\mathcal{F}_t)_{t \geq 0}$. 

\item The c\`adl\`ag, complete filtration  generated by the process $(N_t=\mathbf{1}_{\{\tau\leq  t\}})_{t\geq 0}$ and $(\mathcal{F}_t)_{t \geq 0}$ is denoted 
$(\bar{\mathcal{F}}_t)_{t \geq 0}$. 
\end{itemize}

\end{notation}

Under Assumption \ref{ass:a2}, it is well known by the Feynman-Kac formula (see, for example, \cite[Chapter 5, \S 7.B]{MR1121940}) that the backward Cauchy problem \eqref{eq:pde_fT} has at most one bounded classical solution $\phi(s,y)$. Under Assumptions  \ref{ass:a2} and \ref{ass:uniqueFK}, we denote this solution by $f^T(s,y)$. We will prove the following result.

\begin{theorem}
\label{thm:main_ii_0}
Fix $T> 0$ and let $X = (X_t)_{t \in [0,T]}$ and $\rho \colon \R^d \to \R_+$ satisfy Assumptions \ref{ass:a2} and \ref{ass:uniqueFK}. 
\begin{enumerate}
    \item \label{item:main_ii_0_i}
Under $\widetilde\PP_\nu^T$, the law of $X_0$ is given by the $T$-tilted law $\nu^T$ defined as
\begin{equation}
\label{eq:mu0T}
\diff \nu^T(x) \coloneq \frac{f^T(0,x)}{\int f^T(0,y) \diff \nu(y)} \diff \nu(x).
\end{equation}
 \item \label{item:main_ii_0_ii}
There exists a $\widetilde{\PP}_{\nu}^T$-$(\bar{\mathcal{F}}_t)_{t\in [0,T]}$ Brownian motion $(\widetilde{B}_t)_{t \in [0,T]}$ such that $(X,N)$ satisfies
\begin{equation}
\label{eq:tildedynamics_x_intro}
\begin{split}
\widetilde\PP_\nu^T\text{-a.s. } \forall t \in [0,T],  \quad \diff X_t =&  \diff \widetilde{B}_t + b(X_t)\diff t + \nabla \ln f^T(t,X_t) \diff t ,\\
\diff N_t =& \,  0 .
\end{split}
\end{equation}
 \item \label{item:main_ii_0_iii} For every $t \in [0,T]$, the function $x\mapsto \HH(\widetilde{\PP}_{x}^{T-t}|\PP_{x}^{T-t}) $ is differentiable and the drift  $\nabla \ln f^T(t,x)$ in \eqref{eq:tildedynamics_x_intro} satisfies
\begin{equation}\label{eq:driftentrop}
 \ln  f^T(t,x) = -  \HH(\widetilde{\PP}_{x}^{T-t}|\PP_{x}^{T-t}).   
 \end{equation}
\end{enumerate}
\end{theorem}

 The statement of Theorem \ref{thm:main_ii_0} is possibly known, at least partially or at a heuristic level, but we could not find it in full detail in the literature. Our proof, given in Section \ref{subsec:proofs1.2y1.3}, is based on an application of a general form of Girsanov's theorem to a specific realization of the c\`adl\`ag augmented  Markov process \eqref{eq:absorbedXN}, which is detailed in Section \ref{subsec:prelim}.

 \begin{remark} Some  facts about this result are worth noting: 

 \begin{itemize}
\item
 In general, both the dynamics and the initial law of the original  absorbed process $(X_t,N_t)_{t\in [0,T]}$  are affected by the conditioning. 
 \item The function $u^T(t,x) \coloneq \ln f^T(t,x)$ is well-known to solve the following Hamilton-Jacobi-Bellman equation, as can be easily deduced from \eqref{eq:pde_fT}:
 \[\begin{cases}
\partial_t u^T + \LL^X u^T + \frac12 \abs{\nabla u^T}^2 = \rho, \quad (t,x)\in  [0,T)\times  \mathbb{R}^d\\
u^T(T,x) = 0,  \quad x\in \mathbb{R}^d. 
\end{cases}\]
 Its gradient $\nabla \ln f^T(t,x)$ appearing in the SDE \eqref{eq:tildedynamics_x_intro} can be seen as an analog of the {\it F\"ollmer drift} in Theorem \ref{thm:follmer} for the setting of entropy minimization under distributional constraints discussed in Section \ref{subsubsec:the_gaussian_follmer_process}. 
\item Notice, however, that the actual change of measure required for the (variational) F\"ollmer process interpretation to hold acts on the pair $(X,N)$ and not just on $X$, so the suppression of the killing jump should also be considered as a F\"ollmer drift-type effect at the level of Poisson-driven processes (as dealt with for instance in \cite{MR3904394, MR4864872, MR4989825}). See also the next result and remark, which complete the picture that connects the function $\ln f^T(t,x)$ and the pathwise entropy. 
\end{itemize}
 \end{remark}

\begin{corollary}\label{cor:valueentropy} In the same setting of Theorem \ref{thm:main_ii_0}, we have 
\begin{align*}
\HH(\widetilde{\PP}_{\nu}^T|\PP_{\nu}^T) & = \widetilde{\EE}_{\nu}^T \left[\int_0^T \frac12\abs{\nabla \ln f^T(s,X_s)}^2\diff s + \int_0^T\rho(X_s)\diff s\right] +  \HH(\nu^T| \nu)
\\  & = -  \ln \EE_{\nu}^T[f^T(0,X_0)].
\end{align*}
\end{corollary}

See also Section \ref{subsec:proofs1.2y1.3} for the proof of this result. 

\begin{remark}
Corollary \ref{cor:valueentropy} serves in the present setting as an analog of the entropy formula on the Wiener space from Theorem \ref{thm:follmer} \ref{item:iii:thm:follmer}, and similarly  measures the minimal cost of enforcing the respective distributional constraint.  Here, however, this cost is accounted for by both the change of drift $b\mapsto b + \nabla \ln f^T(s,\cdot) $ of $X$ \textit{and} the suppression of killing  $\rho\mapsto 0$ of $N$ in \eqref{eq:tildedynamics_x_intro}. In fact, a standard Girsanov argument on $X$ leading to the drift change in \eqref{eq:tildedynamics_x_intro} and the initial law tilting \eqref{eq:mu0T} would result only in the entropic cost 
\begin{equation}\label{eq:costwithoutkilling} \widehat{\EE}^T_{\nu} \left[\int_0^T \frac12\abs{\nabla \ln f^T(s,X_s)}^2\diff s\right] +  \HH(\nu^T| \nu)
\end{equation}
with $\widehat{\EE}^T_{\nu}$ denoting the expectation under the law of $X$ after such change of measure (see Remark \ref{rem:entropies} at the end of Section \ref{subsec:proofs1.2y1.3} for details). In turn, even in the case of constant $\rho(x)\equiv \bar{\rho}$,
 in which $\nabla \ln f^T(s,\cdot) \equiv 0$ (cf. equation \eqref{eq:def_fT}), the entropy formula of Corollary \ref{cor:valueentropy} accounts for the corresponding cost of suppressing the killing: $\HH(\widetilde{\PP}_{\nu}^T|\PP_{\nu}^T)=\bar{\rho} T$. 
\end{remark}

\section{Conditional contractivity under asymptotic strict convexity}
\label{sec:cond_contractivity}

Recall that the  $1$-Wasserstein distance $\W_1$ between two probability measures $\eta, \eta' \in \mathcal{P}_1(\R^d)$ is defined as 
\[\W_{1}(\eta, \eta') \coloneq \inf_{\pi \in \Pi(\eta, \eta')} \int_{\R^d \times \R^d} \abs{x-y} \diff \pi(x,y)\] 
where $\Pi(\eta, \eta') \coloneq \{\pi \in \mathcal{P}(\R^d \times \R^d) : \proj_1(\pi) = \eta \text{ and } \proj_2(\pi) = \eta'\}$. 

Leaving the question of existence of a QSD $\alpha$ for later discussion, we next study conditions on the drift term $b$ of the process \eqref{eq:dynamics_x} and the killing rate $\rho$, for the convergence \eqref{eq:QSDconvergence} to take place at an exponential rate in the $1$-Wasserstein distance $\W_1$,  for some general class of initial conditions $\nu$. To that end, we will leverage the pathwise approach to contractivity of diffusion semigroups introduced by Eberle \cite{MR2843007, eberle16reflection} on the basis of reflection couplings, as well as on the theory of propagation of asymptotic log-concavity along  Hamilton-Jacobi flows, developed in the recent work of Chaintron, Conforti, and Eichinger \cite{chaintron25propagation}, which relies on related pathwise ideas. 

Central to our approach will be the stochastic flow associated with the SDE \eqref{eq:tildedynamics_x_intro} and defined in the following result, a detailed proof of which is omitted: 

\begin{lemma}\label{lemma:SDEcond_well}
Suppose Assumptions \ref{ass:a2} and \ref{ass:uniqueFK} hold, let $W$ be a standard $d$-dimensional Brownian motion in some given probability space and  fix $T>0$.  For each $x\in \R^d$  there is a pathwise unique strong solution $(X^{x,T}_t)_{t\in [0,T]}$ to the SDE
 \begin{equation}
\label{eq:SDEcond}
 \diff X^{x,T}_t = \diff W_t + b(X^{x,T}_t)\diff t + \nabla \ln f^T(t,X^{x,T}_t) \diff t,  \quad X_0^{x,T} =x.
\end{equation}
\end{lemma} 
 Indeed, weak existence is provided by Theorem \ref{thm:main_ii_0} with $\nu=\delta_x$ and the rest follows from the locally Lipschitz character in $y$ of the functions $b(y)$ and $\nabla \ln f^T(t,y)$, the latter uniformly on each interval $[0,t]\subset [0,T)$. 
 
We will henceforth work under an additional structural assumption, which places us in the reversible case with a smooth potential and a sufficiently regular killing rate.

\begin{restatable}{assumption}{asslangevin}
\label{ass:langevin}
$~$
\begin{enumerate}
    \item The drift $b\colon \R^d \to \R^d$ has the form $b = -\nabla U$ for some smooth potential $U \colon \R^d \to \R$ with derivatives having polynomial growth and such that $\int_{\R^d}e^{-2U}\diff x < +\infty$ and $\nabla U$ is Lipschitz. In particular, the process $X$ is symmetric with respect to the Gibbs measure $\diff \mu \propto e^{-2U} \diff x$.
    
    \item The killing rate $\rho \colon \R^d \to \R_+$ is smooth, nonnegative, and has derivatives with polynomial growth.
\end{enumerate}
\end{restatable}

\begin{remark}
In particular, Assumption \ref{ass:langevin} implies the former Assumptions \ref{ass:a2} and \ref{ass:uniqueFK}.
\end{remark}

The key notion on which the works  \cite{eberle16reflection}  and \cite{chaintron25propagation} are based is the {\it weak convexity profile} of a $C^2$ function  $V \colon \R^d \to \R$, denoted by $\kappa_V\colon \R_{>0} \to \R$,  which is defined as
\[\forall r > 0, \quad \kappa_V(r) \coloneq \inf\left\{\frac{\langle \nabla V(x) - \nabla V(y), x-y \rangle}{\abs{x-y}^2} : \abs{x-y} = r\right\}.\]
This object recaptures the standard notions of convexity and strong convexity of the potential $V$, as $\inf_{r>0}  \kappa_V(r)  \geq 0$ and $\inf_{r>0}  \kappa_V(r)  \geq a$ for some $a>0$, respectively. Its interest, however, is that it provides a way of measuring convex or strongly convex behavior occurring possibly only at infinity.

More precisely, a $C^2$ function $V \colon \R^d \to \R$ is said to be {\it strictly asymptotically convex } if $\liminf_{r \to +\infty} \kappa_V(r) > 0$. The class of strictly asymptotically convex functions is of course much larger than that of strongly convex functions, and includes all functions $V$ that are regular inside a compact set of $\R^d$ and strongly convex on its complement. See Section \ref{sec:weak_convexity_Eberle} for further discussion on the weak convexity profile and strict asymptotic convexity.  

We now introduce a set of conditions required for our next main result on the dynamics of the conditioned process to hold, and then state this result.

\begin{restatable}{assumption}{asseffectivepotential}
\label{ass:effective_potential}
Define the effective potential as
\[\forall x \in \R^d, \quad U_{\mathrm{eff}}(x)\coloneq \frac12 \abs{\nabla U(x)}^2 - \frac12 \Delta U(x)\]
and for a $C^2$ function $V\colon \R^d \to \R$, let $\kappa_V \colon \R_{>0} \to \R$ be its weak convexity profile.
\begin{enumerate}
    \item The function $r \mapsto r \kappa_{U_{\mathrm{eff}} + \rho}(r)$ is uniformly lower-bounded:
    $$\inf_{r > 0} r \kappa_{U_{\mathrm{eff}} + \rho}(r) > -\infty.$$
    \item The potential $U_{\mathrm{eff}} + \rho$ is strictly asymptotically convex:
    $$\liminf_{r \to +\infty} \kappa_{U_{\mathrm{eff}} + \rho}(r) > 0.$$
    \item The potential $U$ is strictly asymptotically convex:
    $$\liminf_{r \to +\infty} \kappa_{U}(r) > 0$$
    and such that $\int_0^1 r(\kappa_U(r))^-\diff r < +\infty$.
\end{enumerate}
\end{restatable}

\begin{theorem}
\label{thm:contraction_dynamics_intro}
Let Assumptions \ref{ass:langevin} and \ref{ass:effective_potential} hold. Then there exist constants $c^*,C^*>0$ such that for all $T>0$ and all $t\in [0,T]$, the process \eqref{eq:SDEcond} satisfies
    \begin{equation}
    \label{eq:contraction_dynamics_intro_1}
    \W_1(\Law(X_t^{x,T}),\Law(X_t^{y,T}))\leq C^* e^{-c^*t} |x-y|.
    \end{equation}
    Consequently, denoting by $X_t^{X_0,T}$ the process \eqref{eq:SDEcond} when started at a random position $X_0$ independent of $W$, and writing  $\Law^{\nu}(X_t^{X_0,T})$ and $\Law^{\eta}(X_t^{X_0,T})$ for the laws of $X_t^{X_0,T}$ when $X_0$ respectively has the laws $\nu$ and $\eta$, we have, for all $T>0$ and all $t\in [0,T]$,
    \begin{equation}
    \label{eq:contraction_dynamics_intro_2}
    \W_1(\Law^{\nu}(X_t^{X_0,T}),\Law^{\eta}(X_t^{X_0,T}))\leq C^* e^{-c^*t} \W_1(\nu,\eta).
    \end{equation}
\end{theorem} 

Our aim now is to apply Theorem \ref{thm:contraction_dynamics_intro} in the framework provided by Theorem \ref{thm:main_ii_0}, that is,  to the particular case of initial data given by $T$-tilted distributions defined as in \eqref{eq:mu0T}. Let us first generalize that notion as follows.

\begin{definition}
For $T > 0$ and $t \in [0,T]$, we denote by $\nu_t^T$ the law at time $t$ of the  process $X$  in \eqref{eq:dynamics_x} started with initial law $\nu$, conditioned on $\{\tau>T\}$. We call this distribution the $T${\it -tilted law of} $X$ {\it at time} $t$.
\end{definition}

That is, $\nu_t^T$  is the law at time $t\in [0,T]$ of the weak solution to the SDE \eqref{eq:tildedynamics_x_intro} provided by Theorem \ref{thm:main_ii_0}, whose initial condition  is  $\nu^T_0= \nu^T$, the $T$-tilted law defined in \eqref{eq:mu0T}. 

Notice that  $\alpha\in {\cal P}(\R^d)$ is a QSD as defined in \eqref{eq:QSDdef} if and only if $\alpha_T^T=\alpha$ for all $T\geq 0$,  and the convergence  \eqref{eq:QSDconvergence}  in a given sense is tantamount to $\nu_T^T\xrightarrow[T\to +\infty]{} \alpha$. We can now deduce directly from the previous results the following corollary.

\begin{corollary}\label{coro:contractionconseq}
Let Assumptions \ref{ass:langevin} and \ref{ass:effective_potential} hold and let $\nu, \eta \in {\cal P}_1(\R^d)$. Then, for all $T>0$ and all $t\in [0,T]$,
\begin{equation}\label{eq:contractconditionedtilted}
    \W_1(\nu_t^T,\eta_t^T)\leq C^* e^{-c^*t} \W_1(\nu^T,\eta^T).
\end{equation}
In particular, if $\eta=\alpha$ is a QSD, then $\W_1(\alpha,\nu_T^T)\leq C^* e^{-c^*T} \W_1(\alpha^T,\nu^T)$ for all $T>0$. 
\end{corollary} 

Consequently, if the families $(\alpha^T)_{T>0}$ and $(\nu^T)_{T>0}$ have uniformly bounded first moments, then for some constant $K(\nu)>0$, we obtain an exponential convergence bound of the form
$$\W_1(\alpha,\nu_T^T)\leq K(\nu) e^{-c^*T} \quad \mbox{for all } T>0.$$ 
 Sufficient conditions implying this type of bounds will be provided by Theorem \ref{theo:domain_of_attraction_W1} in the next section.

The proof of Theorem \ref{thm:contraction_dynamics_intro} along with some prerequisites is given in Section \ref{sec:weak_convexity_Eberle}.

\section{Exponential QSD convergence in Wasserstein distance } 
\label{sec:qsd}

We begin this section by revisiting the Feynman-Kac semigroup associated with $X$ and $\rho$ and formulate general conditions that ensure the existence of a unique QSD $\alpha$ and the convergence \eqref{eq:QSDconvergence} for a large class of initial distributions $\nu$, first in a suitable weak sense. In Theorem \ref{theo:domain_of_attraction_W1} we will then refine those conditions and apply results  from the previous section to obtain exponential convergence to  $\alpha$  in  $1$-Wasserstein distance.

\begin{definition}
\label{def:feynman_kac_semigroup}
We denote by $(\PPP_t^\rho)_{t \geq0}$ the \textit{Feynman-Kac semigroup} associated with the Markov process $X$ and the killing rate $\rho \colon \R^d \to \R_+$, which acts on sufficiently integrable functions $g \colon \R^d \to \R$ by 
\[\forall t \geq 0, \forall x \in \R^d, \quad \PPP^\rho_t g(x) \coloneq \EE_x\left[g(X_t) e^{-\int_0^t \rho(X_s)\diff s}\right].\]
\end{definition}

\noindent The infinitesimal generator $\LL^\rho$ associated with $(\PPP_t^\rho)_{t \geq0}$ acts on suitable functions $g \colon \R^d \to \R$ via
\[\LL^\rho g = \LL^X g - \rho g.\]
Notice that for any  $\alpha \in {\cal P}(\R^d)$,  
$$ \langle \alpha_T^T, g \rangle = \frac{\langle \alpha, \PPP^\rho_T g \rangle}{\langle \alpha, \PPP_T^\rho \1 \rangle} $$
and that $\alpha \in {\cal P}(\R^d)$ is a QSD in the sense of \eqref{eq:QSDdef} if and only if  $\langle \alpha, g \rangle = \frac{\langle \alpha, \PPP^\rho_T g \rangle}{\langle \alpha, \PPP_T^\rho \1 \rangle} $
for all $g$ measurable (or continuous) and  bounded. We say in that case that $ \alpha$ is a QSD for  $(\PPP_t^\rho)_{t \geq0}$. 

We also recall the notion of {\it domain of attraction of a QSD}, in forms that are suited to our present setting.

\begin{definition}
\label{def:domain_of_attraction} Let $\alpha \in {\cal P}(\R^d)$ be a QSD for  $(\PPP_t^\rho)_{t \geq0}$.  
\begin{itemize}
\item  We say that $\nu \in {\cal P}(\R^d)$ is in the {\it weak domain of attraction} of $\alpha$ if 
$ \nu_T^T \xrightarrow[T \to +\infty]{\text{}} \alpha$  
in the weak topology of probability measures. 
\item Moreover,  if $\alpha \ll \mu$  with  $\frac{\diff \alpha}{\diff \mu} \in L^q(\mu)$ for some $q\in [1,\infty]$, we say that $\nu \in {\cal P}(\R^d)$ is in the {\it weak-$L^q(\mu)$ domain of attraction} of $\alpha$ if  $\nu_T^T \ll \mu$  with  $\frac{\diff \nu^T_T}{\diff \mu} \in L^q(\mu)$ for all $T\geq 0$ and  
$ \frac{\diff \nu^T_T}{\diff \mu} \xrightarrow[T \to +\infty]{\text{}} \frac{\diff \alpha}{\diff \mu} $ 
weakly in $L^q(\mu)$. 

\item Last, if $\alpha \in {\cal P}_1(\R^d)$ we say that $\nu \in {\cal P}_1(\R^d)$ is in the {\it $1$-Wasserstein domain of attraction} of $\alpha$ if $\nu_T^T \in {\cal P}_1(\R^d)$ for all $T\geq 0$ and $\nu_T^T \xrightarrow[T \to +\infty]{\text{}} \alpha$ 
with respect to the $1$-Wasserstein topology.
\end{itemize}
\end{definition}

\begin{remark} Clearly,  both the weak-$L^q(\mu)$ domain of attraction and the $1$-Wasserstein domain of attraction of $\alpha$ are subsets of its weak domain of attraction. However, for  $\nu$ in the weak domain of attraction of $\alpha$ the condition  $\alpha,\nu \ll \mu$ with $\frac{\diff \alpha}{\diff \mu}, \frac{\diff \nu}{\diff \mu} \in L^q(\mu)$ is not enough for $\nu$ to be in the weak-$L^q(\mu)$ domain of attraction of $\alpha$, and  the condition $\alpha,\nu \in {\cal P}_1(\R^d)$ is not enough for $\nu$ to be  in the $1$-Wasserstein domain of attraction of $\alpha$ either.
\end{remark}

The following notion gathers key asymptotic properties of the Feynman-Kac semigroup.

\begin{definition}
\label{def:fk_ergodicity}
Let $p \in [2, +\infty]$. We call  the Feynman-Kac semigroup $(\PPP_t^{\rho})_{t\geq 0}$ {\it weakly $p$-ergodic} if the following hold:
\begin{enumerate}
\item there exist $\phi_0 \in L^p(\mu)$ and $\lambda_0 > 0$ such that $\phi_0 > 0$ $\mu$-a.s., $\langle \mu, \phi_0^2 \rangle = 1$, and
\[\forall t \geq 0, \quad \PPP^\rho_t \phi_0 = e^{-\lambda_0 t} \phi_0, \quad \mu \mbox{-a.s.};\]
\item for every $t \geq 0$, $\PPP_t^{\rho}$ is a linear bounded operator $L^{p}(\mu)\to L^{p}(\mu)$; and 
\item for every $g \in L^{p}(\mu)$,
\[\lim_{t \to +\infty} e^{\lambda_0 t} \PPP_t^\rho g =  \phi_0 \langle \phi_0, g \rangle_{\mu}\mbox{ weakly-* in } L^p(\mu).\]
\end{enumerate} 
The function $\phi_0$ is called a {\it ground state} for the Feynman-Kac semigroup $(\PPP_t^{\rho})_{t\geq 0}$ and $\lambda_0$ is the {\it decay rate}.
\end{definition}

The next result establishes a general relationship between Definition \ref{def:fk_ergodicity} and the existence of a unique QSD $\alpha$. It also provides sufficient conditions for a distribution $\nu$ to belong to the weak-$L^q(\mu)$ domain of attraction of $\alpha$.

\begin{theorem}
\label{theo:domain_of_attraction}
Let $p\geq 2$ and let Assumption \ref{ass:langevin} hold. The following two  statements are equivalent:
\begin{enumerate}
    \item[(a)] For every $t \geq 0$, $\PPP_t^{\rho}$ is a linear bounded operator $L^{p}(\mu)\to L^{p}(\mu)$ and there exist $\lambda_0 > 0$ and $\phi_0 \in L^p(\mu)$ such that $\phi_0 > 0$ $\mu$-a.s., $\langle \mu, \phi_0^2 \rangle = 1$ and 
    $  \phi_0\langle \mu, \phi_0 \rangle   = \lim_{t \to +\infty} e^{\lambda_0 t} \PPP_t^\rho \mathbf{1} $
     weakly-$*$ in $L^p(\mu)$.  Moreover,  the probability measure \begin{equation}\label{eq:alphadef} \diff \alpha \coloneq \phi_0 / \langle \mu, \phi_0 \rangle \diff \mu 
    \end{equation}
    is a QSD for $(\PPP_t^\rho)_{t\geq 0}$, and every
$\nu\in\mathcal P(\R^d)$ such that $\nu\ll\mu$ with
$\frac{\diff\nu}{\diff\mu}\in L^q(\mu)$ for some
$q\geq p/(p-1)$ belongs to the weak-$L^{p/(p-1)}(\mu)$
domain of attraction of $\alpha$.

    \item[(b)] The Feynman-Kac semigroup $(\PPP_t^{\rho})_{t\geq 0}$ is  weakly $p$-ergodic with associated ground state $\phi_0 \in L^p(\mu)$ and decay rate $\lambda_0>0$.     
\end{enumerate}
Last, if one of the two previous points (a) or (b) holds, one has $\lambda_0 = -\frac{\ln \langle \alpha, \PPP_T^\rho \1 \rangle }{T}$ for all $T>0$ and $\alpha$  is  the unique QSD for $(\PPP_t^\rho)_{t \geq0}$ 
such that $\alpha \ll \mu$ with $\frac{\diff \alpha}{\diff \mu} \in L^{p/(p-1)}(\mu)$.

\end{theorem}

In Examples \ref{ex:l2_theory}, \ref{ex:lp_theory}, and \ref{ex:linfty_theory} below, we discuss known general frameworks in which Theorem \ref{theo:domain_of_attraction} applies, and thus weak $p$-ergodic holds, respectively with $p = 2$, $2\leq p<+\infty$ and $p = +\infty$.

\begin{example}[$L^2$ spectral theory]
\label{ex:l2_theory}
Consider $L^2(\mu)$, which we endow with its natural inner product $(g,h) \mapsto \langle g, h\rangle_\mu \coloneq \int_{\R^d}gh \diff \mu$. Recall for $x \in \R^d$ the notation $U_{\mathrm{eff}}(x)= \frac12 \abs{\nabla U(x)}^2 - \frac12 \Delta U(x)$. Let Assumption \ref{ass:langevin} hold, and additionally assume that $\rho \not\equiv 0$ and that the effective potential is confining:
\[\lim_{\abs{x}\to +\infty} U_{\mathrm{eff}}(x) + \rho(x) = +\infty.\]
Then it is known that the operator $\mathcal{L}^\rho$ is negative and essentially selfadjoint in $L^2(\mu)$ \cite[Corollary 10.2]{MR2218016} and that it has a compact resolvent \cite[Theorem XIII.67]{MR493421}. It thus has discrete spectrum and a complete set of eigenfunctions.  More precisely,  there exists an orthonormal sequence of eigenfunctions $(\phi_k)_{k \geq 0} \subseteq L^2(\mu)$ with respective positive eigenvalues $(\lambda_k)_{k \geq 0}$ with $0 < \lambda_0 \leq  \lambda_1 \leq \lambda_2 \leq \cdots$ satisfying
\[\forall k \geq 0, \quad \mathcal{L}^\rho \phi_k = - \lambda_k \phi_k\]
and $\phi_0 > 0 \, \mu$-a.s. Note that for any suitable test function $g \colon \R^d \to \R$,
\[\PPP_t^\rho g = \sum_{k\geq 0} e^{-\lambda_k t} \langle g, \phi_k\rangle_\mu \phi_k,\]
where the convergence of the series is in $L^2(\mu)$ sense. In particular,
\[\PPP_t^\rho \phi_0 = e^{-\lambda_0 t} \phi_0\]
and for all $g\in L^2(\mu)$:
\begin{equation}
\label{eq:wealsonote}
e^{\lambda_0 s} \PPP_s^\rho g \xrightarrow[s\to +\infty]{L^2(\mu)} \langle g, \phi_0\rangle_\mu \phi_0,
\end{equation}
strongly, so the Feynman-Kac semigroup is weakly $2$-ergodic according to Definition \ref{def:fk_ergodicity}.
\end{example}

\begin{example}[A Lyapunov-based criterion]
\label{ex:lp_theory}
Let $2 \leq p < +\infty$ and recall the augmented process $\bar{X}$ defined in \eqref{eq:absorbedXN}. Theorem 3.5 in \cite{MR4546021} applied to the hard-killed process $\bar{X}$ provides a Lyapunov-based criterion to get the weak $p$-ergodicity property for $2 \leq  p < +\infty$. Indeed, if assumption (F) on page 11 of \cite{MR4546021} holds, and additionally we require that $\psi_1^{1/p}$  has at most polynomial growth, where $\psi_1 \colon \R^{d+1} \to [1, +\infty)$ is the Lyapunov function given by (F), and moreover that the Gibbs measure $\mu$ has finite exponential moments (which is the case whenever our Assumption \ref{ass:effective_potential} holds; see Lemma \ref{lemma:exponential_moments} below), then condition (a) in Theorem \ref{theo:domain_of_attraction} is verified as a consequence of Theorem 3.5 in \cite{MR4546021} applied to $\bar{X}$; hence the Feynman-Kac semigroup is weakly $p$-ergodic in this case.
\end{example}

\begin{example}[Coupling-based $L^\infty$ theory]
\label{ex:linfty_theory}
Let $p = +\infty$. Similarly to Example \ref{ex:lp_theory}, consider $\bar{X}$ the extended process defined in \eqref{eq:absorbedXN} and let Assumption (A) in \cite{MR3449390} hold (see page 2 therein). In light of Proposition 2.3 and Theorem 2.1 in \cite{MR3449390}, condition (a) in Theorem \ref{theo:domain_of_attraction} is satisfied, so we deduce that under this assumption the Feynman-Kac semigroup is weakly $\infty$-ergodic.
\end{example}

In the present setting, we can prove the following result, which is a key ingredient of Theorem \ref{theo:domain_of_attraction_W1}.

\begin{proposition}
    \label{prop:4+5weakergo}
Suppose  Assumptions \ref{ass:langevin} and  \ref{ass:effective_potential} hold. For all $p \in [2, +\infty)$, the Feynman-Kac semigroup is weakly $p$-ergodic with rate $\lambda_0>0$ and ground state $\phi_0\in \bigcap_{r\in [2,+\infty)} L^r(\mu)$ not depending on $p$.
\end{proposition}

Recall that  $c^*, C^*>0$ are the constants given in Theorem \ref{thm:contraction_dynamics_intro}. We can now state our main result.


\begin{theorem}\label{theo:domain_of_attraction_W1}

  Suppose  that Assumptions \ref{ass:langevin} and \ref{ass:effective_potential} hold. The probability measure $\alpha$  defined in \eqref{eq:alphadef} is in $  {\cal P}_1(\mathbb{R}^d)$ and its $1$-Wasserstein  domain of attraction  contains any $\nu \in \mathcal{P}(\R^d)$ of the form $\nu=\delta_x$ for some $x\in \mathbb{R}^d$,  or  such that  $\nu \ll \mu$ with $\frac{\diff \nu}{\diff \mu} \in L^{q} (\mu)$ for some $q> 1$.

 Moreover, the families $(\alpha^T)_{T>0}$ and $(\nu^T)_{T>0}$ are bounded in $  {\cal P}_1(\mathbb{R}^d)$ and we have the exponential convergence bound in $1$-Wasserstein distance: 
 $$\W_1(\alpha,\nu_T^T)\leq K(\nu) e^{-c^*T} \quad \mbox{for all } T>0, $$
 with $K(\nu)>0$ a constant depending on $\alpha$ and $\nu$. Last, for each $C' >C^*$, there is some large enough $T_{\nu}>0$  depending on 
$\alpha$ and $\nu$, such that 
 $$\W_1(\alpha,\nu_T^T)\leq  C' e^{-c^*T} \W_1(\alpha^{\infty},\nu^{\infty})<\infty  \quad \mbox{for all } T>T_{\nu}, $$
 where $ \alpha^{\infty},\nu^{\infty} \in  {\cal P}_1(\mathbb{R}^d) $ are the probability measures defined by $\langle \alpha^\infty, g \rangle \coloneq \frac{\langle \alpha, \phi_0 g \rangle}{\langle \alpha, \phi_0\rangle}$ and similarly for $\nu^{\infty}$. 
\end{theorem}

The proofs of the statements in this section are deferred to Section \ref{subsec:proofs_fk}.

\section{Q-process and  Wasserstein exponential ergodicity}
\label{sec:q_process}

We end by studying the law of the process $X$ conditioned to never be absorbed, or the Q-process. More precisely, for $T, S \geq 0$, we consider the limiting behavior of the law $\widetilde{\PP}_\nu^{T+S}$  when $S \to +\infty$, which can be interpreted as the law of the process conditioned to never be absorbed. When that convergence holds, we say that the limiting object is the Q-process. The next result accounts for this fact under our assumptions. Let us write $\bar{\mathbb{D}}_T \coloneq \mathbb{D}([0,T];\R^{d+1})$. 

\begin{proposition}
\label{prop:q_process}
Let Assumptions \ref{ass:langevin} and \ref{ass:effective_potential} hold and let $\phi_0 \in L^p(\mu)$ and $\lambda_0 > 0$ be the ground state and its corresponding decay rate. Then, for every $T \geq 0$ and $\nu \ll \mu$ with $\frac{\diff \nu}{\diff \mu} \in L^{p/(p-1)}(\mu)$, there exists $\QQ^T_\nu \in \mathcal{P}(\bar{\mathbb{D}}_T)$ such that
\[\widetilde{\PP}^{S+T}_\nu \circ \bar{X}_{[0,T]}^{-1} \xrightarrow[S \to +\infty]{\text{weakly}} \QQ_\nu^T.\]
The measure $\QQ_\nu^T$ is the law on $\bar{\mathbb{D}}_T$ of the Q-process with initial distribution $\diff\nu^\infty \propto \phi_0\diff \nu$. Moreover, we have $\QQ^T_\nu \ll \PP_\nu^T$ with
\[\frac{\diff \QQ^T_\nu}{\diff \PP_\nu^T} = \frac{e^{\lambda_0 T}\phi_0(X_T) \1_{\{N_T=0\}}}{\int_{\R^d} \phi_0(x) \diff \nu(x)}.\]
\end{proposition}

We will see that, in an equivalent way, the law of the Q-process is determined by the Markovian semigroup, symmetric with respect to the probability measure $\diff \beta \propto \phi_0\diff \alpha$, that is associated with the SDE
\begin{equation}
\label{eq:dynamics_q_process_intro}
\diff Y_t = \diff B_t - \nabla(U-\ln \phi_0)(Y_t) \diff t. 
\end{equation}
This opens the door to studying its trend to its equilibrium $\beta$ by pathwise methods in a similar line as utilized to prove Theorem \ref{thm:contraction_dynamics_intro}. Along this line, we are also able to state exponentially fast convergence of the Q-process towards its equilibrium $\beta$.

\begin{theorem}
\label{prop:q_process_exp_fast_intro}
Let Assumptions \ref{ass:langevin} and \ref{ass:effective_potential} hold and additionally assume that $U_{\mathrm{eff}} + \rho$ has a bounded Hessian. Let $Y = (Y_t)_{t \geq 0}$ be the dynamics of the Q-process, which solves the stochastic differential equation \eqref{eq:dynamics_q_process_intro}, assume that $\Law(Y_0) \in \mathcal{P}_1(\R^d)$, and let $\diff \beta \propto \phi_0\diff \alpha$ be its invariant distribution. Then there exist constants $C_{\rm Q},c_{\rm Q}>0$, independent of $\eta$, such that
\[\forall t\geq0,\quad \W_1(\Law(Y_t),\beta) \leq C_{\mathrm Q}e^{-c_{\mathrm Q}t} \W_1(\eta,\beta).\]
\end{theorem}

The proofs of Proposition \ref{prop:q_process} and Theorem \ref{prop:q_process_exp_fast_intro} are deferred to Section \ref{subsec:proofs_qprocess}.

\section{Proofs of the main results}
\label{sec:proofs}

\subsection{Proof of Theorem \ref{thm:main_i}}
\label{subsec:proof_main_i}
\begin{proof}[Proof of Theorem \ref{thm:main_i}]
Let $\mathbb{Q}\ll \PP^T_{\nu}$ be a probability measure such that $\Law_{\mathbb{Q}}(N_T)= \delta_0 = \Law_{\widetilde \PP_{\nu}^T}( N_T )$
and write $D_T \coloneq  \frac{\diff \mathbb{Q}}{\diff \PP^T_{\nu}}$. By Jensen's inequality,
$$ \EE_{\nu}[\varphi(D_T)]=\EE_{\nu}[\EE[\varphi(D_T) | N_T] ] \geq \EE_{\nu}[\varphi(\EE [D_T\vert N_T] ) ]. $$
On the other hand, for $f \colon \N \to \R$,
$$\frac{1}{\PP_{\nu}(N_T=0 )}\EE_{\nu}[f(N_T)\mathbf{1}_{N_T=0}] = f(0)=\EE^{\mathbb{Q}}[f(N_T)] = \EE_{\nu}[f(N_T)D_T],   $$
which implies that
$$ \EE[D_T|N_T ]=\frac{\mathbf{1}_{N_T =0 }}{\PP_{\nu}( N_T = 0) } = \frac{\diff  \widetilde{\PP}_{\nu}^T}{\diff \PP^T_{\nu}} \ \mathbb{P}_{\nu} \mbox{-a.s.}$$
Hence
$$ \EE[\varphi(D_T)] \geq \EE\left[\varphi\left(\frac{\diff \widetilde{\PP}_{\nu}^T}{\diff \PP^T_{\nu}} \right) \right]. $$
That is, $\widetilde{\PP}_{\nu}^T$  minimizes $\mathbb{H}^T_\varphi$ over the set $\{\mathbb{Q} \in \mathcal{P} (\mathbb{D}([0,T],\R^d)): \Law_{\mathbb{Q}}( N_T )= \delta_0 \}$.
\end{proof}

\begin{remark}
Since the process $X$ is absorbed at $\partial$ upon hitting that set for the first time, for any $S>0$ the law of $N_T$ under $\widetilde{\mathbb{P}}_{\nu}^{T, T+S} \coloneq \widetilde{\mathbb{P}}_{\nu}^{T+S}\circ {X}_{[0,T]}^{-1}$ is equal to $\delta_0$.  Consequently, we have 
$$\mathbb{H}^T_\varphi (\widetilde{\PP}_{\nu}^T)\leq \mathbb{H}_\varphi^T (\widetilde{\mathbb{P}}_{\nu}^{T,T+S}).$$

\end{remark}

\subsection{Preliminaries on softly killed  diffusion process}\label{subsec:prelim}

We will work with  a specific realization of the Markov process  $X$ killed at state-dependent rate $\rho(X_t)$, defined by using an auxiliary Poisson point measure. This construction is certainly  well-known and of course is equivalent to the simpler one using a time-changed independent exponential variable, but it facilitates the use of stochastic analysis techniques to establish Theorem \ref{thm:main_ii_0} and Corollary \ref{cor:valueentropy}.  We notice  that this construction  applies to an arbitrary Markov process $X$.

\begin{notation}
$~$
\begin{itemize}
\item We enlarge the probability space  $(\Omega, \mathcal{F},  \mathbb{P})$ on which the process $X$ is defined, by means of an independent Poisson point measure ${\mathcal N}$ on $\R_+ \times \R_+$ with intensity the Lebesgue measure and a random variable $N_0$ in $\N$ independent from it and from $X$. The enlarged space is still denoted $(\Omega, \mathcal{F},  \mathbb{P})$

\item We write $(\mathcal{G}_t)_{t \geq 0}$ for the filtration in this space  generated for each $t\geq 0 $ by $\mathcal{F}_t$, $N_0$ and the process $({\cal N}([0,t]\times A): A\subset \mathbb{R}_+$ Borel).
\end{itemize}
\end{notation}

Let us define a random intensity process  $\lambda = (\lambda_t)_{t \geq 0}$ by 
\begin{equation}
\label{eq:def_intensity}
\forall t \geq 0, \quad \lambda_t \coloneq \rho(X_t)
\end{equation}
and an associated counting process
$N^\lambda = (N^\lambda_t)_{t \geq 0}$ by
\begin{equation}
\label{eq:def_general_counting_process}
\forall t \geq 0, \quad N^\lambda_t \coloneq N_0+\int_0^t\int_{\R_+} \1_{\{z \leq \lambda_s\}} \diff \mathcal{N}( s, z).
\end{equation}
From the properties of ${\cal N}$, the process $(X,N^{\lambda})$ taking values in $\mathbb{R}^d\times \mathbb{N}\subset \mathbb{R}^{d+1}$ can easily be seen to be Markov. 
An application  of It\^o's formula shows its infinitesimal generator is given by
$$\LL^{(X,N^\lambda)} f(x,n) = \LL^X f(\cdot,n)(x) + \rho(x)[f(x,n+1)-f(x,n)], $$
for functions $f\colon \mathbb{R}^d\times \mathbb{N} \to \R$ such that for all $n \in  \mathbb{N}$, $f(\cdot, n) $ is in the domain of $\LL^X$, the infinitesimal generator  of the process $X$.

Set now $\tau =  \inf \{t\geq 0 : N_t^\lambda \geq 1\}$ and define a process $N$  by 
$$N=(N_t)_{t\geq 0}\coloneq (N^{\lambda}_{t\wedge \tau})_{t \geq 0} $$
which satisfies $N_t=\mathbf{1}_{\{\tau\leq t\}}$ $\PP$-a.s. on the event $\{N_0=0\}$. 
The following simple result, proved here for completeness, confirms that the  Markov process 
$(X_{t\wedge \tau}, N_{t})_{t \geq 0}$ just defined is a concrete realization of the process $(\bar{X}_t)_{t \geq 0}$ given by \eqref{eq:absorbedXN}.

\begin{lemma}
\label{lemma:killing_via_intensity}
For every $\nu\in {\cal P}(\R^d)$ the  $({\cal G}_t)_{t\geq 0}$-stopping time  $\tau$ satisfies, for all $T>0$, 
$${\PP}_{\nu\otimes \delta_0}(\tau > T | X_{[0,T]}) = \exp\left(-\int_0^T \rho(X_s)\diff s\right),$$ 
where ${\PP}_{\nu\otimes \delta_0}$ stands for the law of $(X,N)$ initialized with law  $\nu\otimes \delta_0$.
\end{lemma}

\begin{proof}
By the definition  \eqref{eq:def_general_counting_process} of $N_t^\lambda$, the events $\{\tau > T\}$ and $\{N_T=0\}$ coincide ${\PP}_{\nu\otimes \delta_0}$-a.s. Moreover, 
conditionally on $X_{[0,T]}$, the random variable $N_T^\lambda$ is Poisson of parameter $\int_0^T \lambda_t \diff t$. It follows that 
$$ {\PP}_{\nu\otimes \delta_0}(\tau > T | X_{[0,T]})  = {\PP}_{\nu\otimes \delta_0}(N_T^\lambda = 0|X_{[0,T]}) = \exp\left(-\int_0^T\lambda_t\diff t \right) = \exp\left(-\int_0^T\rho(X_t)\diff t \right).$$
\end{proof}

Recall that we use the shorthands $\PP^T_{\nu}  $ for the joint law of  $(X,N)_{[0,T]}$ under ${\PP}_{\nu\otimes \delta_0}$ and $\widetilde{\PP}^T_{\nu}$  for its law conditioned on $\{\tau>T\} =\{N_T=0\} $.  Notice that the filtration $( \bar{{\cal F}}_t)_{t\geq 0}$  defined in Section \ref{subsec:condiffkillfoll} by 
$$ \bar{{\cal F}}_t= {\cal F}_t\vee \sigma(N_s : s \leq t) $$ satisfies $ \bar{{\cal F}}_t\subset  {\cal G}_t$ for all $t\geq 0$. Let us set 
$$ Z_T \coloneq \frac{\diff  \widetilde{\PP}_{\nu}^T}{\diff  \PP^T_{\nu}}=\frac{\mathbf{1}_{N_T =0 }}{\PP^T_{\nu}( N_T = 0) } \mbox{ and }Z_t \coloneq \EE_{\nu}^T[Z_T\vert \bar{{\cal F}}_t] \mbox{ for }t\in [0,T].$$
The following result provides a (first) more explicit form of the martingale  $(Z_t)_{t\in [0,T]}$.

\begin{lemma}
\label{lemma:computation_Z}
For each $t \in [0,T]$ we have $\PP^T_{\nu}$-a.s : 
\[Z_t = \frac{\1_{\{N_t = 0\}}}{\PP_{\nu}^T(N_T=0)}\EE_{\nu}^T\left[e^{-\int_t^{T}\rho(X_s)\diff s}\middle| X_t \right].\]
\end{lemma}

\begin{proof}
Since the process $(N_t)_{t \in [0,T]}$ is $\PP_{\nu}^T$-a.s. non-decreasing, we have $\{N_T = 0\} \subseteq \{N_t = 0\}$ $\PP_{\nu}^T$-a.s., hence $\EE_{\nu}^T[\1_{\{N_T=0\}}|\mathcal{G}_t] = \EE_{\nu}^T[\1_{\{N_T - N_t=0\}}|{\mathcal{G}}_t] \1_{\{N_t = 0\}} $. The latter is $\bar{\cal F}_t$-measurable, which gives us
\[\EE_{\nu}^T[\1_{\{N_T=0\}}|\bar{\mathcal{F}}_t] = \EE_{\nu}^T[\1_{\{N_T - N_t=0\}}|\bar{\mathcal{F}}_t] \1_{\{N_t = 0\}}=\PP_{(X_t,0)}[N_{T-t}=0] \1_{\{N_t = 0\}}\]
by the Markov property of the process $(X,N)$ and with $\PP_{(y,0)}$ denoting its law when starting from $(y,0)$.   On the other hand, by Lemma \ref{lemma:killing_via_intensity}, 
\begin{equation*} 
\EE_{(y,0)}[\1_{\{N_{T-t}=0\}} ] = \EE_{y}[\EE_y [\1_{\{N_{T-t}=0\}} \vert X_{[0,T-t]}]  ] = \EE_{y}\left[e^{-\int_0^{T-t}\rho(X_s)\diff s} \right] .
\end{equation*}
Since $\EE\left[e^{-\int_t^{T}\rho(X_s)\diff s}\middle| X_t \right]= \EE_{X_t}\left[e^{-\int_0^{T-t}\rho(X_s)\diff s} \right] $ by the Markov property of $X$, the statement follows. 

\end{proof}

\begin{remark}\label{rk:value_fT(0,x)}
Under Assumption \ref{ass:uniqueFK}, by the Feynman-Kac representation \cite[Chapter 5, \S 7.B]{MR1121940}, the function $f^T\colon [0,T]\times \R^d \to \R_+$ solution to \eqref{eq:pde_fT} is given by
\begin{equation}
\label{eq:def_fT}
\forall (t, y) \in [0,T]\times \R^d, \quad f^T(t,y)\coloneq \EE\left[e^{-\int_t^{T}\rho(X_s)\diff s}\middle| X_t=y\right].
\end{equation}
In particular, by Lemma \ref{lemma:killing_via_intensity}, we have
\begin{equation}\label{eq:f0=PNT0}
    f^T(0,x) = \EE_x\left[e^{-\int_0^{T}\rho(X_s)\diff s}\right] = \PP_x^T(N_T = 0).
    \end{equation}
    By homogeneity of the Markov process $X$, for all $0\leq t\leq T$ and $x\in \R^d$ we also have
$$ f^T(t,x)=f^{T-t}(0,x).$$
Hence, defining a function $f\colon \R_+\times \R^d\to \R$ by
$$ f(t,x)= f^{t}(0,x)$$
we readily see that $f$ is the unique bounded solution of the {\it forward} Cauchy problem: 

\begin{equation}
\label{eq:pde_ffor}
\begin{cases}
\partial_s \phi (s,y)  = \LL^X \phi (s,y) - \rho (y) \phi(s,y), \quad (s,y)\in [0,+\infty)\times \mathbb{R}^d\\
\phi(0,y) = 1, \quad  y\in \mathbb{R}^d 
\end{cases}
\end{equation}
and  $f^T(t,x)=f(T-t,x)$. We will, however, stick to the notation $f^T$, which makes explicit the time horizon of the conditioning being considered.

\end{remark}

\begin{lemma}
\label{lemma:computation_fT}
Fix $T > 0$. Let the function $f^T\colon [0,T]\times \R^d \to \R_+$ be the unique solution of \eqref{eq:pde_fT} and notice that its gradient is locally bounded on $[0,T)\times \R^d$. Let $(L_t)_{t \in [0,T)}$ be the  $({\cal F}_t)_{t\in [0,T)}$-$\PP_{\nu}^T$-local martingale given by
\begin{equation}
\label{eq:def_L}
\forall t \in [0,T), \quad L_t \coloneq \int_0^t \nabla \ln f^T(s,X_s) \cdot \diff B_s.
\end{equation}
Then, $\PP_{\nu}^T$-a.s. $L_t$ and $[L]_t$  have (finite) limits  $L_T$ and $[L]_T$ as $t\nearrow T$, and we have
\begin{equation}
\label{eq:almost_doleans_dade}
f^T(t,X_t) = f^T(0,X_0) \exp\left(L_t - \frac12[L]_t\right) \exp\left(\int_0^t\rho(X_s)\diff s\right) \quad \forall t \in [0,T].
\end{equation}
\end{lemma}

\begin{proof}
 Let $t \in [0,T)$. By Itô's formula and \eqref{eq:pde_fT}, we have
\begin{align*}
f^T(t,X_t) &= f^T(0,X_0) + \int_0^t(\partial_s + \LL^X)f^T(s,X_s)\diff s + \int_0^t \nabla f^T(s,X_s) \cdot \diff B_s\\
&= f^T(0,X_0) + \int_0^t \rho(X_s) f^T(s,X_s) \diff s + \int_0^t \nabla f^T(s,X_s) \cdot \diff B_s. 
\end{align*}
By continuity of $f^T$ on $[0,T]\times \R^d$, local boundedness of $\rho$ and the previous,  the limit $L_T$  of $L_t$ as $t\nearrow T$ a.s. exists. Moreover by \eqref{eq:def_fT} we have $f^T(t,y)>0$  on $[0,T]\times \R^d$. Therefore we can apply It\^o's formula (together with a localization argument to deal with the logarithm)  and get for all $t\in [0,T)$: 
\begin{align}
\label{eq:log_fT}
\notag \ln f^T(t,X_t) &= \ln f^T(0,X_0) + \int_0^t\nabla \ln f^T(s,X_s) \cdot \diff B_s\\ \notag
&\quad \quad \quad \quad \quad \quad \quad \quad \quad + \int_0^t\left(\rho(X_s) - \frac12 \abs{\nabla \ln f^T(s,X_s)}^2\right)\diff s\\
&= \ln f^T(0,X_0) + L_t - \frac12[L]_t + \int_0^t\rho(X_s)\diff s.
\end{align}
 Thus, $[L]_T<+\infty$ and the proof is finished. 
\end{proof}

We can now write the martingale $(Z_t)_{t \in [0,T]}$ as a Doléans-Dade exponential (see e.g. Theorem 37, \cite[Chapter II, \S 8]{MR2273672}) and apply a suitable version of Girsanov's theorem, from which the proofs of Theorem \ref{thm:main_ii_0} and Corollary \ref{cor:valueentropy} will follow. 

\begin{theorem}
\label{thm:doleans-dade_girsanov}
The  process $M = (M_t)_{t \in [0,T]}$   defined by
\begin{equation}
\label{eq:martingale_M}
\forall t \in [0, T], \quad M_t \coloneq L_t - \widetilde{N}_t,
\end{equation}
where $L=(L_t)_{t \in [0,T]}$ is defined by \eqref{eq:def_L} and $\widetilde{N}_t \coloneq N_t - \int_0^{t\wedge \tau} \rho(X_s)\diff s$ for $t \in [0,T]$, is a $(\bar{{\cal F}}_t)_{t\in [0,T)}$-$\PP_{\nu}^T$-local martingale. Moreover,  
$\PP^T_{\nu}$-a.s., 
\begin{equation}
\label{eq:z_doleans_dade}
\forall t \in [0,T], \quad Z_t = \frac{ \PP_{X_0}^T(N_T=0)}{\PP_{\nu}^T(N_T=0)}\mathcal{E}(M)_t 
\end{equation}
with
\begin{align*}
\mathcal{E}(M)_t \coloneq &   \exp\left(L_t - \frac12[L]_t\right) \exp\left(\int_0^t\rho(X_s)\diff s\right) \1_{\{N_t = 0\}} \\
= &  \exp(M_t - \frac12 [ M]^c_t) \prod_{0<s \leq t} (1 + \Delta M_s) \exp(-\Delta M_s) 
\end{align*}
where $[ M]^c=[M^c]$ denotes the continuous part of the quadratic variation $[ M]$ of $M$. 
Consequently, for any locally square-integrable
$(\bar{{\cal F}}_t)_{t\in [0,T]}$-$\PP_{\nu}^T$- local martingale $M'$,  
the process $M'- \langle M',M\rangle$, with $\langle M',M\rangle$ the predictable covariation of  $M'$ and $M$ under $\PP_{\nu}^T$,  is a
$(\bar{\cal F}_t)_{t\in [0,T]}$-$\widetilde{\PP}_{\nu}^T$-local martingale.
\end{theorem}

\begin{proof}
The process $L$ being a   $({\cal F}_t)_{t\in [0,T)}$-$\PP_{\nu}^T$-local martingale 
independent of the Poisson measure ${\cal N}$, it is also a $({\cal G}_t)_{t\in [0,T)}$-$\PP_{\nu}^T$-local martingale. On the other hand,  with the notation in \eqref{eq:def_general_counting_process}, the equality
$\widetilde{N}_t = \int_0^t\int_{\R_+}  \1_{\{z \leq \rho(X_s), N_{s-}=0 \}}  \left[ \mathcal{N}(\diff s, \diff z)- \diff z \diff s \right] $ holds $\PP_{\nu}^T$-a.s. and we see that $\widetilde{N}$ is a compensated $({{\cal G}}_t)_{t\in [0,T]}$-${\PP}_{\nu}^T$- martingale. Moreover, $L$ and $\widetilde{N}$ are adapted to $(\bar{{\cal F}}_t)_{t\in [0,T]}$, so the first claim follows.

Let now $t \in [0,T]$. By Remark \ref{rk:value_fT(0,x)}, we have $f^T(0,x) = \PP_x^T(N_T=0)$, so in light of  Lemmas \ref{lemma:computation_Z}  and   \ref{lemma:computation_fT}, the $(\bar{{\cal F}}_t)_{t\in [0,T]}$-martingale $(Z_t)_{t\in [0,T]}$ satisfies 
\begin{equation}
\label{eq:eq_aux_proof_exponential_1}
Z_t \frac{\PP_{\nu}^T(N_T=0)}{\PP_{X_0}^T(N_T=0)} = \1_{\{N_t = 0\}}\frac{f^T(t,X_t) }{f^T(0,X_0)} =  \exp\left(L_t - \frac12[L]_t\right) \exp\left(\int_0^{t}\rho(X_s)\diff s\right) \1_{\{N_t = 0\}}
\end{equation}
which provides formula \eqref{eq:z_doleans_dade}. We next obtain the  second asserted expression (in Doléans-Dade form) for $\mathcal{E}(M)_t$. Note that the second exponential in the last term  of \eqref{eq:eq_aux_proof_exponential_1} satisfies
\begin{equation}
\label{eq:eq_aux_proof_exponential_2}
\1_{\{N_t = 0\}}\exp\left(\int_0^{t}\rho(X_s)\diff s\right) = \exp\left(-\widetilde{N}_t + N_t\right)\1_{\{N_t = 0\}}.
\end{equation}
Since $(N_t)_{t \in [0,T]}$ is a pure-jump process, we have
\begin{equation}
\label{eq:eq_aux_proof_exponential_3}
[M^c]_t= [M]^c_t = [L]_t.
\end{equation}
We also note that
\begin{equation}
\label{eq:eq_aux_proof_exponential_4}
\exp(N_t) = \prod_{0< s\leq t} \exp(\Delta N_s) = \prod_{0< s\leq t} \exp(-\Delta M_s),
\end{equation}
and that, since $\PP_{\nu}$-a.s.  $ \1_{\{N_0=0\}}=1$,
\begin{equation}
\label{eq:eq_aux_proof_exponential_5}
\1_{\{N_t=0\}} = \prod_{0< s \leq t} \1_{\{\Delta N_s = 0\}} = \prod_{0< s \leq t} (1- \1_{\{\Delta N_s = 1\}}) = \prod_{0 < s \leq t}(1 + \Delta M_s).
\end{equation}
Hence, putting together \eqref{eq:eq_aux_proof_exponential_2}, \eqref{eq:eq_aux_proof_exponential_3}, \eqref{eq:eq_aux_proof_exponential_4}, and \eqref{eq:eq_aux_proof_exponential_5} into \eqref{eq:eq_aux_proof_exponential_1}, we obtain the desired  expression for \eqref{eq:z_doleans_dade}.  The last statement follows from Girsanov's theorem for c\`adl\`ag semimartingales in the absolutely continuous setting (see e.g. Theorem 41 in \cite[Chapter III, \S 8]{MR2273672}). Indeed, the process $\langle M',M\rangle$
exists as both $M'$ and $ M$ are locally square-integrable $(\bar{{\cal F}}_t)_{t\in [0,T]}$-$\PP_{\nu}^T$-local martingales, and moreover one has
$$M'_t- \int_0^t \frac{1}{Z_{s-}} \diff \langle Z,M '\rangle_s = M'_t- \langle M',M \rangle_t $$
as follows from  the identities
$$ \diff Z_t= Z_{t-} \diff M_t \mbox{ and }\diff \langle Z, M'\rangle_t= Z_{t-} \diff \langle M,M'\rangle_t ,   $$ ensured by  Theorem 37, \cite[Chapter II, \S 8]{MR2273672} and the fact that $Z_{t-}$ is a $(\bar{{\cal F}}_t)_{t\in [0,T]}$ predictable process. 


\end{proof}

\begin{remark}
    Notice that under $\widetilde{\PP}_{\nu}^T$, the process $N$ is deterministic since identically null. Hence the completed sigma-fields ${\cal F}_t$ and $\bar{{\cal F}}_t$ coincide for each $t\in [0,T]$ in the probability space $(\Omega,{\cal F},\widetilde{\PP}_{\nu}^T)$. 
\end{remark}

\subsection{Proofs of Theorem \ref{thm:main_ii_0} and Corollary \ref{cor:valueentropy}}\label{subsec:proofs1.2y1.3}

\begin{proof}[Proof of Theorem \ref{thm:main_ii_0}]

Part (i) is immediate by taking $t=0$ in \eqref{eq:z_doleans_dade} and using \eqref{eq:f0=PNT0}. 
Part (ii) follows from Theorem  \ref{thm:doleans-dade_girsanov}, more precisely by applying its conclusion to the continuous martingale $M'=B$, noting that $[M,M']=[L,B]$ in that case,  and using Lévy's characterization theorem for Brownian motion. 
Last, part (iii) comes from the identity
$$\HH(\widetilde{\PP}_{x}^{T-t}|\PP_{x}^{T-t})=  - \ln  f^{T-t}(0,x) = - \ln  f^T(t,x),  $$
which corresponds  to the 
particular case $\nu=\delta_x$ in Corollary \ref{cor:valueentropy} (see Remark \ref{rk:HPPlnf} below) and the fact that $- \ln  f^{T-t}(0,x) = - \ln  f^T(t,x) $ as noted in  Remark \ref{rk:value_fT(0,x)}.
\end{proof}

\begin{proof}[Proof of Corollary \ref{cor:valueentropy}]
  By Lemma \ref{lemma:computation_Z} and the identity \eqref{eq:def_fT}, and since $f^T(T,y)=1$,
we have 
\begin{align*} \1_{\{N_T=0\}} \ln Z_T & =  \1_{\{N_T=0\}} ( \ln f^T(T,X_T) - \ln\PP_{\nu}^T(N_T=0) ) \\
&= - \1_{\{N_T=0\}}  \ln \PP_{\nu}^T(N_T=0).
\end{align*}
Taking expectation under $\widetilde{\PP}_{\nu}^T$ proves the second asserted equality. This also shows that  $\EE^T_{\nu}(\varphi(Z_T))<\infty$
where $\varphi(r) = r \ln r \1_{\{r \geq 0\}} + \infty \1_{\{r < 0\}}$, which implies that  $(\varphi(Z_t))_{t \in [0,T]}$ is a (uniformly integrable)  $(\bar{{\cal F}}_t)_{t\in [0,T]}$-$\PP_{\nu}^T$-submartingale. For the remaining equality, we adapt a classic argument to our setting.  By Theorem \ref{thm:doleans-dade_girsanov}, there is   a sequence $(T_n)$ of  $(\bar{{\cal F}}_t)_{t\in [0,T)}$-stopping times such that $T_n\nearrow T$ and which localizes the $\widetilde{\PP}_{\nu}^T$-local martingale   $L - [L]$. Using \eqref{eq:eq_aux_proof_exponential_1} with $t=T_n$, and then multiplying by $\1_{\{N_T=0\}}$, we get
\begin{align*}
\1_{\{N_T=0\}} \ln Z_{T_n} &= \1_{\{N_T=0\}} \left(L_{T_n} - \frac12[L]_{T_n} + \int_0^{T_n}\rho(X_s)\diff s\right)\\
& \quad \quad + \1_{\{N_T=0\}} \left( \ln \PP_{X_0}^T(N_T=0)- \ln {\PP_{\nu}^T(N_T=0)} \right)\\
&= \1_{\{N_T=0\}} \left(L_{T_n} - [L]_{T_n} + \frac12[L]_{T_n} + \int_0^{T_n}\rho(X_s)\diff s \right)  \\
& \quad \quad + \1_{\{N_T=0\}} \left( \ln \PP_{X_0}^T(N_T=0)- \ln {\PP_{\nu}^T(N_T=0)} \right). 
\end{align*}
Since  $\widetilde{\PP}_{\nu}^T (N_T=0)=1$,  we have $\widetilde{\EE}_{\nu}^T(\1_{\{N_T=0\}} \ln Z_{T_n}) ={\EE}_{\nu}^T(\varphi(Z_{T_n})) \longrightarrow {\EE}_{\nu}^T(\varphi(Z_{T}))=\HH(\widetilde{\PP}_{\nu}^T|\PP_{\nu}^T)$ as $n\to \infty$. Taking expectation in  the above display,  using monotone convergence  on the r.h.s. while letting $n\to \infty$, we finally get
\begin{align*}
\HH(\widetilde{\PP}_{\nu}^T|\PP_{\nu}^T) -  \widetilde{\EE}_{\nu}^T \left[\ln\left( \frac{ \PP_{X_0}^T(N_T=0)}{\PP_{\nu}^T(N_T=0)}\right) \right]  & = \widetilde{\EE}_{\nu}^T [\ln Z_T]-  \widetilde{\EE}_{\nu}^T \left[\ln\left( \frac{ \PP_{X_0}^T(N_T=0)}{\PP_{\nu}^T(N_T=0)}\right) \right]  \\
&= \widetilde{\EE}_{\nu}^T\left[\frac12[L]_T + \int_0^T\rho(X_s)\diff s\right]\\
&= \widetilde{\EE}_{\nu}^T\left[\int_0^T \frac12\abs{\nabla \ln f^T(s,X_s)}^2\diff s + \int_0^T\rho(X_s)\diff s\right]
\end{align*}
as required.

\end{proof}

\begin{remark}\label{rk:HPPlnf}
Taking $\nu=\delta_x$  in Corollary \ref{cor:valueentropy}, and noting that $(\delta_x)^T=\delta_x $ (in the notation of \eqref{eq:mu0T}), we get that, for all $x\in \R^d$,
$$\HH(\widetilde{\PP}_{x}^T|\PP_{x}^T)  = \widetilde{\EE}_{x}^T \left[\int_0^T \frac12\abs{\nabla \ln f^T(s,X_s)}^2\diff s + \int_0^T\rho(X_s)\diff s \right] = - \ln  f^T(0,x).$$
\end{remark}

\begin{corollary}
\label{cor:rel_entropy_Q_PTT'}
For every $T\geq t>0$, 
\begin{align*}
\HH(\widetilde{\PP}_{\nu}^{t,T}|\PP_{\nu}^{t}) & = \widetilde{\EE}_{\nu}^{t,T} \left[\int_0^t \frac12\abs{\nabla \ln f^{T}(s,X_s)}^2\diff s + \int_0^t\rho(X_s)\diff s\right] +  \HH(\nu^{T}| \nu)
\\  & =\widetilde{\EE}_{\nu}^{t,T} ( \ln f^T(t,X_t))  -  \ln \PP_{\nu}^T(N_{T}=0).
\end{align*}
\end{corollary}
\begin{proof}
We proceed as in the proof of Corollary \ref{cor:valueentropy}. For the first formula, we start  from 
\eqref{eq:eq_aux_proof_exponential_1} with general $t\leq T$ to similarly obtain
\begin{equation*}
\HH(\widetilde{\PP}_{\nu}^{t,T}|\PP_{\nu}^t) -  \widetilde{\EE}_{\nu}^T \left[\ln\left( \frac{ \PP_{X_0}^T(N_T=0)}{\PP_{\nu}^T(N_T=0)}\right) \right] = \widetilde{\EE}_{\nu}^{t,T}\left[\int_0^t \frac12\abs{\nabla \ln f^T(s,X_s)}^2\diff s + \int_0^t\rho(X_s)\diff s\right],
\end{equation*}
 which is equivalent to the claimed equality. For the second formula, again using Lemma \ref{lemma:computation_Z} and identity \eqref{eq:def_fT}, we obtain
\begin{equation*} \1_{\{N_t=0\}} \ln Z_t  =  \1_{\{N_t=0\}} ( \ln f^T(t,X_t) - \ln\PP_{\nu}^T(N_T=0) )
\end{equation*}
and conclude taking expectation under $\widetilde{\PP}_{\nu}^{T}$.
\end{proof}

\begin{remark}\label{rem:entropies}
    Since the dynamics of the process $N$ is trivialized by the conditioning, a natural question is whether the change of measure induced by it on the path space of $X$ results in the same value as in Corollary \ref{cor:valueentropy} for the corresponding entropy. More precisely, letting 
$$\widehat{Z}_T\coloneq \EE_{\nu}^T\left[\frac{\diff  \widetilde{\PP}_{\nu}^T}{\diff  \PP^T_{\nu}}\bigg| {\cal F}_T\right]=\frac{e^{-\int_0^T\rho(X_t) \diff t}}{\PP^T_{\nu}( N_T = 0) }\quad \mbox{and}\quad \widehat{\PP}^T_{\nu} \coloneq \widehat{Z}_T \cdot \PP^T_{\nu} \vert_{ {\cal F}_T},$$
we ask for the relation between the quantities
$\widehat{\EE}^T_{\nu}(\ln \widehat{Z}_T) =\EE_{\nu}^T\left( \widehat{Z}_T \ln \widehat{Z}_T \right)$  and  $  \HH(\widetilde{\PP}_{\nu}^T|\PP_{\nu}^T)  $. Since $f^T(T,\cdot)=1$, by formula \eqref{eq:almost_doleans_dade} with $t=T$ we get
    $$\exp\left(-\int_0^T\rho(X_s)\diff s\right) = f^T(0,X_0) \exp\left(L_T - \frac12[L]_T\right). $$
    In other words, the change of measure on $X$ associated with the density $\widehat{Z}_T$ is exactly the one that results in the drift change \eqref{eq:tildedynamics_x_intro} and the initial law tilting \eqref{eq:mu0T}. A classic computation (or a rewriting of computations in the proof of Corollary \ref{cor:valueentropy} with $\rho\equiv 0$) then shows that $ \widehat{\EE}^T_{\nu}(\ln \widehat{Z}_T)  $ is equal to the expression \eqref{eq:costwithoutkilling}. That is to say, the value of the entropy when computed in the sigma field  ${\cal F}_T$ instead of  $\bar{{\cal F}}_T$ differs exactly  in the cost of controlling the trajectory of $N$ in order to suppress the killing jump, something the process $X$ does not have access to on its own. 
    
\end{remark}

\subsection{Proof of Theorem \ref{thm:contraction_dynamics_intro}}
\label{sec:weak_convexity_Eberle}

The proof of Theorem \ref{thm:contraction_dynamics_intro} will be based on Eberle's approach to the exponentially fast convergence of diffusions under asymptotic convexity of the drift \cite{MR2843007,eberle16reflection}, and on the propagation of this property along the Hamilton-Jacobi flow, recently proven by Chaintron, Eichinger, and Conforti \cite{chaintron25propagation}. We start by recalling the notion of weak convexity profile,  leveraged in \cite{MR2843007} from the context of reflection couplings \cite{lindvall86coupling}.

\begin{definition}
\label{def:weak_convexity_profile}
Let $V \colon \R^d \to \R$ be a $C^2$ function. The \textit{weak convexity profile} of $V$ is the function $\kappa_V \colon \R_{>0} \to \R$ defined by
\[\forall r > 0, \quad \kappa_V(r) \coloneq \inf\left\{\frac{\langle \nabla V(x) - \nabla V(y), x-y \rangle}{\abs{x-y}^2} : \abs{x-y} = r\right\}.\]
\end{definition}

The weak convexity profile $\kappa_V$ captures the \textit{isotropic integrated convexity} of the potential $V\colon \R^d \to \R$ in the following sense: for $a \in \R$ and $r > 0$, one has
\[\kappa_V(r) \geq a \iff \left\{\begin{array}{c}
    \displaystyle \int_0^1 (y-x)^\intercal \nabla^2V(tx+(1-t)y) (y-x) \diff t \geq a \abs{y-x}^2\\
    \displaystyle \forall x,y \in \R^d \text{ with } \abs{y-x}=r.
\end{array}\right.\]
In particular, $V$ is convex if and only if $\kappa_V(r) \geq 0$ for all $r > 0$ and $V$ is $a$-strongly convex for some $a > 0$ if and only if $\kappa_V(r) \geq a$ for all $r > 0$. We next introduce the notion of strict asymptotic convexity, which generalizes the usual notion of strong convexity.

\begin{definition}
\label{def:strict_convex}
Let $V \colon \R^d \to \R$ be a $C^2$ function. We say that $V$ is \textit{strictly asymptotically convex} if $\liminf_{r \to +\infty} \kappa_V(r) > 0$.
\end{definition}

\begin{remark}
The class of strictly asymptotically convex functions is rich: let $V, V_1, V_2 \colon \R^d \to \R$ be $C^2$ functions.
\begin{enumerate}
    \item Obviously, if $\kappa_V \geq \bar{\kappa}$ on $\R_{>0}$ with $\liminf_{r \to +\infty} \bar \kappa > 0$, then $V$ is strictly asymptotically convex.
    
    \item If $V$ is $c$-strongly convex for some $c > 0$, then $V$ is strictly asymptotically convex.
    
    \item A remarkable example beyond strong convexity is the following: if $V$ is regular inside a compact set $A \subset \R^d$ and strongly convex outside $A$, then $V$ is strictly asymptotically convex.

    \item If both $V_1$ and $V_2$ are strictly asymptotically convex, then $V_1+V_2$ is strictly asymptotically convex.
    
    \item If $V$ is strictly asymptotically convex and $h\colon \R^d \to \R$ is $C^1$ and Lipschitz, then $V+h$ is strictly asymptotically convex.
\end{enumerate}
\end{remark}

The interest of strictly asymptotically convex potentials is that they possess nice contractive properties for diffusion processes, as shown in Theorem \ref{thm:eberle} below. Before we unveil the theorem, we need the following preliminary statement, which gathers several results proven in \cite{eberle16reflection}. Its message is that to every strictly asymptotically convex profile we may associate a natural Wasserstein-type distance that is equivalent to the usual $1$-Wasserstein distance.

\begin{proposition}
\label{prop:properties_f_kappa}
Let $\kappa\colon \R_{>0} \to \R$ be a continuous function with $\liminf_{r \to +\infty} \kappa(r) > 0$ and $\int_0^1 r (\kappa(r))^- \diff r < +\infty$, where $(\cdot)^-$ denotes the negative part of a real number. Define $f_\kappa \colon \R_+ \to \R_+$ by
\[\forall r \geq 0, \quad f_\kappa(r) \coloneq \int_0^r \phi_\kappa(s) g_\kappa(s) \diff s,\]
where $\phi_\kappa(r) \coloneq \exp\left(-\frac12 \int_0^r s (\kappa(s))^- \diff s\right)$, $\Phi_\kappa(r) \coloneq \int_0^r \phi_\kappa(s) \diff s$, $R_0 \coloneq \inf\{R\geq 0 : \inf_{r \geq R}\kappa(r) \geq 0\}$, $R_1 \coloneq \inf\{R \geq R_0 : \kappa(r)R(R-R_0) \geq 4 \ \forall r \geq R\}$, and
\[g_\kappa(r) \coloneq 1-\left(\frac12\int_0^{r \wedge R_1} \frac{\Phi_\kappa(s)}{\phi_\kappa(s)}\diff s\right) / \left(\int_0^{R_1}\frac{\Phi_\kappa(s)}{\phi_\kappa(s)}\diff s \right).\]

Then:
\begin{enumerate}
    \item The function $f_\kappa$ is differentiable, concave, and for the constant $C_\kappa \coloneq \phi_\kappa(R_0)/2 \in (0,1)$, we have
\[\forall r \geq 0, \quad C_\kappa r \leq f_\kappa(r) \leq r.\]

    \item \label{item:properties_f_kappa_ii} For $\eta, \eta' \in  \mathcal{P}_1(\R^d)$, we set
\[\W_{f_\kappa}(\eta, \eta') \coloneq \inf_{\pi \in \Pi(\eta, \eta')} \int_{\R^d \times \R^d} f_\kappa(\abs{x-y}) \diff \pi(x,y),\]
where $\Pi(\eta, \eta') \coloneq \{\pi \in \mathcal{P}(\R^d \times \R^d) : \proj_1(\pi) = \eta \text{ and } \proj_2(\pi) = \eta'\}$. Then
\[C_\kappa \W_1(\eta, \eta') \leq \W_{f_\kappa}(\eta, \eta') \leq \W_1(\eta, \eta')\]
where $\W_1$ denotes the standard $1$-Wasserstein distance:
\[\W_{1}(\eta, \eta') \coloneq \inf_{\pi \in \Pi(\eta, \eta')} \int_{\R^d \times \R^d} \abs{x-y} \diff \pi(x,y).\]
\end{enumerate}
\end{proposition}

\begin{remark}
The technical condition $\int_0^1 r (\kappa(r))^- \diff r < +\infty$ in Proposition \ref{prop:properties_f_kappa} controls how $\kappa$ diverges to $-\infty$ as $r \to 0$. For example, when $\kappa$ is uniformly lower-bounded in a neighborhood of $0$, this condition is verified.
\end{remark}

We can now state the main result in \cite{eberle16reflection}. In short, it states that if the convexity profiles of a family of potentials $(V_t)_{t \geq0}$ are minored by a fixed profile $\bar{\kappa}$ not depending on $t\geq 0$ that is strictly positive at infinity, then the law of the solution of the SDE
\begin{equation}
\label{eq:eq_thm_eberle_1_aux}
\diff Y_t = \diff B_t - \nabla V_t(Y_t) \diff t
\end{equation} 
 contracts exponentially fast in the $\W_{f_{\bar{\kappa}}}$ distance from Proposition \ref{prop:properties_f_kappa}. Its proof combines the behavior outside some compact set of the potentials $(V_t)_{t \geq0}$, described by the weak convexity profile $\bar{\kappa}$, with the use of a reflection coupling inside that set. We remark that in the original work \cite{eberle16reflection}, the theorem is proven for a process like \eqref{eq:eq_thm_eberle_1_aux} with a drift that does not depend on the time parameter $t \geq 0$; however, its extension to the time-dependent case is straightforward (see \cite{MR4674060}).

\begin{theorem}[Theorem 1 in \cite{eberle16reflection}]
\label{thm:eberle}
Let $V \colon [0,+\infty)\times \R^d \to \R$, $(t,x)\mapsto V_t(x) $   be a  continuous function with $V_t(\cdot)$ of class $C^2$ for all $t\geq 0$.  Assume a solution $(Y_t)_{t \geq 0}$ to the  SDE
\eqref{eq:eq_thm_eberle_1_aux} exists, 
 and  moreover that there is  $\bar{\kappa} \colon \R_{>0} \to \R$ such that $\liminf_{r \to +\infty}\bar{\kappa}(r) > 0$, $\int_0^1 r (\bar{\kappa}(r))^- \diff r < +\infty$, and $\kappa_{V_t} \geq \bar{\kappa}$ on $\R_{>0}$ for every $t \geq 0$. Let $\eta, \eta' \in \mathcal{P}(\R^d)$ have a finite first moment and for every $t \geq 0$ define $\eta_t$ (resp. $\eta'_t$)  as the law at time $t$ of the process \eqref{eq:eq_thm_eberle_1_aux} with initial condition $Y_0 \sim \eta$ (resp. $Y_0 \sim \eta'$). Then we have 
\begin{equation}
\label{eq:eq_thm_eberle_2}
\forall t \geq 0, \quad \W_{f_{\bar{\kappa}}}(\eta_t, \eta_t') \leq e^{-c_{\bar{\kappa}} t} \W_{f_{\bar{\kappa}}}(\eta, \eta')
\end{equation}
with $c_{\bar{\kappa}} \coloneq \left(\int_0^{R_1} \Phi_{\bar{\kappa}}(s)/\phi_{\bar{\kappa}}(s) \diff s\right)^{-1} > 0$ and  $R_1, \Phi_{\bar{\kappa}}$, and $\phi_{\bar{\kappa}}$ defined in Proposition \ref{prop:properties_f_kappa}.
\end{theorem}

\begin{remark}
\label{rk:contr_w_1}
In light of point \ref{item:properties_f_kappa_ii} in Proposition \ref{prop:properties_f_kappa}, the estimate \eqref{eq:eq_thm_eberle_2} yields the following contraction in $\W_1$:
\[\forall t \geq 0, \quad \W_1(\eta_t, \eta_t') \leq C_{\bar{\kappa}}^{-1} e^{-c_{\bar{\kappa}} t} \W_1(\eta,\eta').\]
\end{remark}

A second key ingredient in the proof of  Theorem \ref{thm:contraction_dynamics_intro} will be the results in \cite{chaintron25propagation} of propagation of asymptotic convexity along the Hamilton-Jacobi flow,  recorded in a way that is adapted to our setting in the following proposition.

\begin{proposition}
\label{prop:conforti}
Let Assumptions \ref{ass:langevin} and \ref{ass:effective_potential} hold.
\begin{enumerate}
    \item There exists $\bar{\kappa} \colon \R_{>0} \to \R$ with $\liminf_{r \to +\infty} \bar\kappa(r) > 0$ and $\int_0^1 r (\bar \kappa(r))^-\diff r < +\infty$, independent of $T$, such that
    \[\forall t \in [0,T], \forall r > 0, \quad \kappa_{U - \ln f^T(t, \cdot)}(r) \geq \bar \kappa(r).\]
    \item Assume additionally that  that $U_{\mathrm{eff}} + \rho$ has a bounded Hessian, and let $\phi_0 \in L^p(\mu)$ be the associated ground state (recall Section \ref{sec:qsd}). Then the potential $U - \ln \phi_0$ is strictly asymptotically convex.
\end{enumerate}
\end{proposition}

\begin{proof}
For point (i), let $a > 0$ sufficiently small with $a < \liminf_{r \to +\infty} \kappa_U(r)$ and $a^2 < \liminf_{r \to +\infty} \kappa_{U_{\rm eff} + \rho}(r)$ and let $b > 0$ sufficiently big such that $b \geq \sup_{r>0}r(a-\kappa_U(r))$ and $ab > \sup_{r >0} r(a^2-\kappa_{U_{\rm eff} + \rho}(r))$. Define the profile $h\colon \R_{>0} \to \R$ by
\[\forall r > 0, \quad h(r) \coloneq ar-b.\]
In particular, $h$ satisfies both conditions (4) in \cite[Theorem 1.2]{chaintron25propagation}, so the initial condition in equation (5) in \cite[Theorem 1.2]{chaintron25propagation} holds for $\varphi \equiv \1$. Since $f^T(t,\cdot)=\PPP_{T-t}^\rho\1$, we obtain by their theorem that
\[\kappa_{U-\ln f^T(t,\cdot)}(r) \geq \bar\kappa(r)\coloneq a-\frac br\]
for every $T>0$, $t \in [0,T]$, and $r>0$. Note that $\bar\kappa$ is continuous on $\R_{> 0}$, $\liminf_{r\to+\infty}\bar\kappa(r)=a>0$, and
\[\int_0^1 r(\bar\kappa(r))^-\diff r = \int_0^1(b-ar)^+\diff r \leq b<+\infty,
\]
which proves point (i).

Point (ii) follows directly from \cite[Theorem 1.3]{chaintron25propagation}.
\end{proof}

\begin{remark}\label{rem:asymptotEntropy}
Notice that, by Theorem  \ref{thm:main_ii_0} (iii), the conclusion of point (i) in Proposition \ref{prop:conforti} equivalently reads
\[\forall T>0, \forall r > 0, \quad \kappa_{U + \HH(\widetilde{\PP}_{\cdot}^{T}|\PP_{\cdot}^{T})}(r) \geq \bar \kappa(r).\] 
\end{remark}

We can now prove Theorem \ref{thm:contraction_dynamics_intro}.

\begin{proof}[Proof of Theorem \ref{thm:contraction_dynamics_intro}]
Defining for each $T>0$ the family of potentials
\[\forall t \geq 0, \quad V_t(x) \coloneq  U(x) - \ln f^T(t\wedge T,x), \quad x\in \mathbb{R}^d,\]
in light of Proposition \ref{prop:conforti}, we can directly apply Theorem \ref{thm:eberle} and Remark \ref{rk:contr_w_1} to conclude.
\end{proof}

\subsection{Proofs of the results in Section \ref{sec:qsd}}\label{subsec:proofs_fk}
We will need some basic facts about 
the Feynman-Kac semigroup $(\PPP_t^\rho)_{t \geq0}$, stated in the following lemma.

\begin{lemma}\label{lemma:basicFK}
Suppose  Assumption \ref{ass:langevin} holds. Then, the semigroup  $(\PPP_t^\rho)_{t \geq0}$  is symmetric with respect to $\mu$.  As a consequence, 
$\PPP_t^{\rho}\colon L^{p}(\mu)\to L^{p}(\mu)$  is a linear bounded operator for $p\in (1,\infty)$  if and only if it is a 
linear bounded operator $L^{q}(\mu)\to L^{q}(\mu)$ for $q=p/(p-1)$. Moreover, if $\nu \ll \mu$  then 
$\nu_T^T \ll \mu$  for all $T\geq 0$ with  $\frac{\diff \nu^T_T}{\diff \mu}= \langle \nu, \PPP^{\rho}_T \mathbf{1} \rangle^{-1} P^{\rho}_T\frac{\diff \nu}{\diff \mu} $,  
and one has $\frac{\diff \nu^T_T}{\diff \mu}  \in L^q(\mu)$ as soon as $\frac{\diff \nu}{\diff \mu} \in L^q(\mu)$ and $\PPP_t^{\rho}: L^{q}(\mu)\to L^{q}(\mu)$ is bounded linear.

\end{lemma}
\begin{proof}
For all $T \geq 0$ and  $g,h \in C_b(\mathbb{R}^d)$ we have
\[\langle \mu, g\PPP_T^\rho h\rangle = \EE_{\mu}\left[g(X_0) h(X_T) e^{-\int_0^T \rho(X_s)\diff s}\right] = \EE_{\mu}\left[g(X_T) h(X_0) e^{-\int_0^T \rho(X_{T-s})\diff s}\right] = \langle \mu, h\PPP_T^\rho g\rangle.\]
Consequently, for all $h\in  C_b(\mathbb{R}^d)\cap  L^{q}(\mu)$ and $g\in  C_b(\mathbb{R}^d)\cap  L^{p}(\mu)$ 
$$ |\langle \mu, g (\PPP_T^\rho h) \rangle |\leq \norm{ h}_{ L^{q}(\mu)}  \norm{\PPP_T^\rho g }_{ L^{p}(\mu)}\leq C  \norm{ h}_{ L^{q}(\mu)}  \norm{ g }_{ L^{p}(\mu)}  $$
and the second statement follows by a density and duality argument. The previous argument and bound also hold for $q=1$  implying that $\nu_T^T \ll \mu$ for all $T\geq 0$ whenever $\nu \ll \mu$, with  $\frac{\diff \nu^T_T}{\diff \mu}= \langle \nu, \PPP^{\rho}_T \mathbf{1} \rangle^{-1}   \PPP^{\rho}_T\frac{\diff \nu}{\diff \mu}$. The last claim is now immediate.

\end{proof}

\begin{proof}[Proof of Theorem \ref{theo:domain_of_attraction}]

In the proof of equivalence of (a) and (b), we will use the fact that for any $\nu \in {\cal P}(\R^d)$, $g \in C_\mathrm{b}(\R^d)$, and $\phi_0 \in L^1(\mu)\cap L^1(\nu)$, one has

\begin{align}\label{eq:nuPTgalpha}
\frac{\langle \nu, \PPP_T^\rho g\rangle}{\langle \nu, \PPP_T^\rho \1\rangle} - \langle \alpha, g\rangle  \notag
&= \frac{\langle\nu, e^{\lambda_0 T} \PPP_T^\rho g - \phi_0 \langle \mu, \phi_0 \rangle  \langle \alpha, g \rangle \rangle}{\langle \mu, \phi_0 \rangle \langle \nu, \phi_0 \rangle + \langle \nu, e^{\lambda_0 T} \PPP_T^\rho \1 - \phi_0 \langle \mu, \phi_0 \rangle \rangle}  \\ & \qquad - \frac{\langle\alpha, g \rangle \langle \nu, e^{\lambda_0 T} \PPP_T^\rho \1 - \phi_0  \langle \mu, \phi_0 \rangle \rangle}{\langle \mu, \phi_0 \rangle \langle \nu, \phi_0 \rangle + \langle \nu, e^{\lambda_0 T} \PPP_T^\rho \1 - \phi_0 \langle \mu, \phi_0 \rangle \rangle}. 
\end{align}

\noindent \underline{(a) $\Rightarrow$ (b)}: Taking $\lambda_0$ and $\phi_0$ as in the statement, we need to prove that properties (i)  and (iii) in 
Definition \ref{def:fk_ergodicity} hold with these objects.  

\noindent (i) Since $\alpha$ in \eqref{eq:alphadef} is a QSD for $ (\PPP_t^\rho)_{t \geq0}$, we
have
$\langle \alpha, \PPP^\rho_T g \rangle= \langle \alpha, g \rangle  \langle \alpha, \PPP_T^\rho \1 \rangle $ for all bounded measurable functions $g$. Moreover, 
it is well known (see e.g. \cite{MR2994898, MR2986807}) that  $\tilde{\lambda}_0 \coloneq -\frac{\ln \langle \alpha, \PPP_T^\rho \1 \rangle }{T}$ does not depend on $T$. Hence, using the reversibility of $\mu$ with respect to $ (\PPP_t^\rho)_{t \geq0}$ we get, for all $t\geq 0$,
$$ \PPP^\rho_t \phi_0 = e^{-\tilde{\lambda}_0 t} \phi_0, \quad \mu \mbox{-a.s.} $$
Now, since from (a) we have
$   \lim_{t \to +\infty} e^{\lambda_0 t} \PPP_t^\rho \mathbf{1}=\phi_0 \langle \mu, \phi_0 \rangle   $
     weakly-$*$ in $L^p(\mu)$, we get, thanks again to the reversibility, 
     $$  e^{\lambda_0 T} \langle \mathbf{1} ,  \PPP_T^\rho  \phi_0 \rangle_{\mu} = e^{\lambda_0 T} \langle \PPP_T^\rho \mathbf{1} , \phi_0 \rangle_{\mu}\xrightarrow[T \to +\infty]{} \langle \mu, \phi_0^2\rangle \langle \mu, \phi_0 \rangle = \langle \mu, \phi_0 \rangle.$$
     Combined with the previous, this implies that 
     $$  e^{(\lambda_0-\tilde{\lambda}_0) T}  \langle \mu, \phi_0 \rangle \xrightarrow[T \to +\infty]{} \langle \mu, \phi_0 \rangle . $$
     Whence, $\lambda_0=\tilde{\lambda}_0$ and property (i) holds. 

\noindent (iii) For any $\nu$ as in statement (a), we have   $\frac{\diff \nu}{\diff \mu} \in L^q(\mu)$ for  $q \coloneq p/(p-1)$. Thanks to this and the convergence $   \lim_{t \to +\infty} e^{\lambda_0 t} \PPP_t^\rho \mathbf{1}=\phi_0 \langle \mu, \phi_0 \rangle   $
     weakly-$*$ in $L^p(\mu)$, for $\lambda_0$ and $\phi_0$ as above the second term on the right hand side of \eqref{eq:nuPTgalpha} goes to $0$ as $T\to +\infty$ for  $g\in L^p(\mu)$. Moreover, the left hand side goes to $0$ for such functions $g$ too, since $\nu$ is in the weak-$L^q(\mu)$ domain of attraction  of $\alpha$. Since 
 \begin{equation}\label{eq:mumumu} \langle \mu, \phi_0 \rangle  \langle \alpha, g \rangle  = \langle \mu,  \phi_0 g \rangle =  \langle g, \phi_0\rangle_{\mu}, 
 \end{equation} we deduce from 
\eqref{eq:nuPTgalpha} that 
$$\langle\nu, e^{\lambda_0 T} \PPP_T^\rho g - \phi_0   \langle  g, \phi_0  \rangle_{\mu} \rangle \xrightarrow[T \to +\infty]{} 0.$$
This readily yields  
$\lim_{t \to +\infty} e^{\lambda_0 t} \PPP_t^\rho g =  \phi_0 \langle \phi_0, g \rangle_{\mu}$  weakly-* in $ L^p(\mu)$.

\noindent \underline{(b) $\Rightarrow$ (a)}: 
The existence of $\lambda_0$ and $\phi_0$ with the required properties is included in points (i) and (iii) of Definition \ref{def:fk_ergodicity},  taking in the latter $g=\1$. Let us now check that $\alpha$  defined as $\diff \alpha \coloneq \phi_0 / \langle \mu, \phi_0 \rangle \diff \mu $  is a QSD. Indeed,
since the measure $\mu$ is reversible for $(\PPP_t^\rho)_{t \geq 0}$  by Lemma \ref{lemma:basicFK}, we have for  any  $g \in C_\mathrm{b}(\R^d)$ that 
\begin{align*}
\langle \alpha_T^T, g \rangle &= \frac{\langle \alpha, \PPP^\rho_T g \rangle}{\langle \alpha, \PPP_T^\rho \1 \rangle} = \frac{\langle \mu, \phi_0 \PPP^\rho_T g \rangle}{\langle \mu, \phi_0 \PPP_T^\rho \1 \rangle} = \frac{\langle \mu, g \PPP^\rho_T\phi_0 \rangle}{\langle \mu, \PPP^\rho_T \phi_0 \rangle} = \frac{e^{-\lambda_0 T}\langle \mu, g \phi_0 \rangle}{e^{-\lambda_0 T}\langle \mu, \phi_0 \rangle} = \langle \alpha, g \rangle.
\end{align*}

Set $r=p/(p-1)$, with the convention $r\coloneq 1$ if $p=+\infty$. Let $\nu\ll\mu$ satisfy
$h\coloneq\frac{\diff\nu}{\diff\mu}\in L^q(\mu)$
for some $q\geq r$, so $h\in L^r(\mu)$. For every $g\in L^p(\mu)$, weak $p$-ergodicity gives
\[\int h\,e^{\lambda_0T}\PPP_T^\rho g\diff\mu \xrightarrow[T\to+\infty]{} \left(\int h\phi_0\diff\mu\right) \left(\int g\phi_0\diff\mu\right).\]
Taking $g=\1$, we also obtain
\[\int h\,e^{\lambda_0T}\PPP_T^\rho\1\,\diff\mu \xrightarrow[T\to+\infty]{} \left(\int h\phi_0\diff\mu\right) \left(\int \phi_0\diff\mu\right)>0.\]
The positivity follows from $\phi_0>0$ $\mu$-a.s.,
$h\geq0$, and $\int h\,\diff\mu=1$. Consequently,
\begin{equation}
\label{eq:convnunualpha}
\langle \nu_T^T,g\rangle = \frac{\int h\,e^{\lambda_0T}\PPP_T^\rho g\diff\mu}{\int h\,e^{\lambda_0T}\PPP_T^\rho\1\diff\mu} \xrightarrow[T\to+\infty]{} \frac{\int g\phi_0\diff\mu}{\int\phi_0\diff\mu} = \langle\alpha,g\rangle
\end{equation}
for every $g\in L^p(\mu)$. By symmetry, the density of $\nu_T^T$ is
\[\frac{\diff\nu_T^T}{\diff\mu} = \frac{\PPP_T^\rho h}{\langle\nu,\PPP_T^\rho\1\rangle}.\]
Since $\PPP_T^\rho$ is a contraction on $L^r(\mu)$, this density belongs to $L^r(\mu)$ for every $T\geq0$. Moreover, $\phi_0\in L^p(\mu)\subseteq L^r(\mu)$, so $\frac{\diff\alpha}{\diff\mu}\in L^r(\mu)$. That is,
\eqref{eq:convnunualpha} proves that
\[\frac{\diff\nu_T^T}{\diff\mu} \xrightarrow[T\to +\infty]{\text{weakly in } L^r(\mu)} \frac{\diff\alpha}{\diff\mu}.\]
Hence $\nu$ belongs to the weak-$L^{p/(p-1)}(\mu)$ domain of attraction of $\alpha$.

For the last assertions, note that we already established that $\tilde{\lambda}_0=\lambda_0$, and that uniqueness of the QSD   follows readily from \eqref{eq:convnunualpha},  taking therein $\nu=\alpha'$ for  $\alpha' \in {\cal P}(\R^d)$ a second QSD satisfying the required conditions of absolute continuity and integrability. 

\end{proof}

Recall now the weak convexity profile $\kappa_U$  from Definition \ref{def:weak_convexity_profile}. We next show that strict asymptotic convexity of the potential $U$ ensures the finiteness of all moments of the invariant distribution $\mu$. We also deduce from it and from some results in \cite{MR2653230} that $\mu$ satisfies a Logarithmic Sobolev Inequality (LSI). These two properties will be relevant in proving Theorem \ref{theo:domain_of_attraction_W1}.

\begin{lemma}
 \label{lemma:exponential_moments}
 Suppose that $U\colon \R^d \to \R$ is of class $C^2$ with Lipschitz gradient and strictly asymptotically convex. Then, for every $\gamma \in \R$, we have $\int_{\R^d} e^{\gamma \abs{x}- 2 U(x)} \diff x < + \infty$. Moreover, the probability measure $\diff \mu(x)\propto  e^{-2 U(x)}\diff  x$ satisfies a LSI. 
 \end{lemma}

\begin{proof}
Let $a,R > 0$ such that $\kappa_U(r) \geq a$ for $r \geq R$. Then, for all $x\in \mathbb{R}^d$ with $|x|\geq R$ we have 
\begin{equation}\label{eq:ULSI}
x\cdot \nabla U(x) \geq a|x|^2  - \abs{\nabla U(0)}\abs{x}.
\end{equation}
This implies, for any unit vector $\theta \in \mathbb{S}^{d-1}$, that
\[\forall r \geq R, \quad \frac{\diff}{\diff r} U(r \theta) \geq ar - \abs{\nabla U(0)}.\]
Hence, there is some $C\in \mathbb{R} $ such that
\[\forall x \in \R^d, \quad U(x) \geq \frac{a}{2} \abs{x}^2 - \abs{\nabla U(0)}\abs{x} + C,\]
which yields the desired exponential integrability. For the second assertion, notice that \eqref{eq:ULSI} implies 
$$x\cdot \nabla U(x) \geq c|x|^2 $$
for some $c>0$. Since moreover by Assumption \ref{ass:langevin} we have $\nabla^2 U\geq K I_d$ for some $K\in \mathbb{R}$, then \cite[Corollary  2.1 (2)]{MR2653230} ensures that $\mu $ satisfies a LSI.
\end{proof}

The following is a key observation for the sequel:
\begin{remark}\label{rem:45confining} Under Assumptions \ref{ass:langevin} and \ref{ass:effective_potential} (ii), by applying the arguments of the previous proof to the effective potential $U_{\mathrm{eff}}+\rho$ instead of $U$, one checks that the potential $U_{\mathrm{eff}} + \rho$ is confining.  As a consequence, the setting of Example \ref{ex:l2_theory} applies in our case, so that $(\PPP_t^{\rho})_{t\geq 0}$ is weakly $2$-ergodic. Moreover, for all $g\in L^2(\mu)$, one has
\begin{equation}
\label{eq:ergoL2strong}
e^{\lambda_0 s} \PPP_s^\rho g \xrightarrow[s\to +\infty]{L^2(\mu)} \langle g, \phi_0\rangle_\mu \phi_0
\end{equation}
where $\phi_0\in L^2(\mu)$ and $\lambda_0>0$ are respectively the corresponding ground state and  decay rate.
\end{remark} 

We can now provide the proof of Proposition  \ref{prop:4+5weakergo}.

\begin{proof}[Proof of Proposition  \ref{prop:4+5weakergo}] Since $\mu$ satisfies a LSI, the semigroup $(\PPP_t)_{t\geq 0}$ of the original process $(X_t)_{t\geq 0}$ is hypercontractive. In particular for each $p>2$ there is $t_p>0$ such that $\PPP_{t_p}$ acts $\PPP_{t_p}\colon L^{2}(\mu)\to L^{p}(\mu)$ continuously, from which  it follows that $\PPP_{t_p}^{\rho}\colon L^{2}(\mu)\to L^{p}(\mu)$ continuously  too.   

Now, by Assumptions  \ref{ass:langevin} and \ref{ass:effective_potential} (ii) and Remark  \ref{rem:45confining}, there exists  $\phi_0\in L^2(\mu)$ a ground state and $\lambda_0>0$ such that 
 for all $ t \geq 0$, $\PPP^\rho_t \phi_0 = e^{-\lambda_0 t} \phi_0, \quad \mu \mbox{-a.s.}$ 
In particular, since $\phi_0 =e^{\lambda_0 {t_p}} \phi_0 \PPP^\rho_{t_p}$ for each $p>2$, we deduce that   $\phi_0\in L^p(\mu)$. Moreover for every $g\in L^p(\mu)$ and $t>0$ large enough one has:
\begin{equation}
e^{\lambda_0 t} \PPP_t^\rho  g - \langle g, \phi_0\rangle_\mu \phi_0 = e^{\lambda_0 t_p} \PPP_{t_p}^\rho\left(e^{\lambda_0 (t-t_p)} \PPP_{t-t_p}^\rho g -   \langle g, \phi_0\rangle_\mu \phi_0 \right).
\end{equation}
Bringing together \eqref{eq:ergoL2strong} and the continuity $\PPP_{t_p}^{\rho}\colon L^{2}(\mu)\to L^{p}(\mu)$ we deduce that \eqref{eq:ergoL2strong} also holds in $L^p(\mu)$. The fact that $\PPP_{t}\colon L^{p}(\mu)\to L^{p} (\mu)$ is linear is direct from Assumption \ref{ass:langevin}. We conclude that the Feynman-Kac semigroup is weakly $p$-ergodic. 
\end{proof}

\begin{proof}[Proof of Theorem \ref{theo:domain_of_attraction_W1}]

By Corollary \ref{coro:contractionconseq} we have
\begin{equation}
\label{eq:proof_contr_aux_1}
\W_1(\nu_T^T, \alpha) = \W_1(\nu_T^T, \alpha_T^T) \leq C^*e^{-c^*T}\W_1(\nu^T, \alpha^T)
\end{equation}
for some constants $C^*,c^* > 0$ independent of $T$. In the case that $\nu=\delta_x$, we have $\nu^T=\delta_x$ for all $T>0$ and it is enough to prove that the first order moments of the family $(\alpha^T)_{T>0}$ are uniformly bounded. That general property does of course not depend on $\nu$ and will be established as a consequence of the next discussion. 

Given $\eta\in \mathcal{P}(\R^d)$ a probability measure such that $\eta \ll \mu$ and $\langle \eta, \phi_0\rangle<+\infty$, denote by $\eta^\infty \in \mathcal{P}(\R^d)$  the probability  measure
\[\forall g \in C_\mathrm{b}(\R^d), \quad \langle \eta^\infty, g \rangle \coloneq \frac{\langle \eta, \phi_0 g \rangle}{\langle \eta, \phi_0\rangle}.\]
Then we have 
\begin{equation}
\label{eq:proof_contr_aux_2}
\W_1(\nu^T, \alpha^T) \leq \W_1(\nu^T, \nu^\infty) + \W_1(\nu^\infty, \alpha^{\infty}) + \W_1( \alpha^{\infty}, \alpha^T).
\end{equation}
Assume now that $\nu \ll \mu$ with $\frac{\diff \nu}{\diff \mu} \in L^q(\mu)$ for some $q>1$, as in the statement (when $\nu=\delta_x$, only the argument below for $\alpha$ is needed, as explained above). Fix $p > \max\{2, q/(q-1)\}$ and let $m \coloneq p/(p-1)$ be its H\"older conjugate, so that $1 < m < q$ and $m < 2 < p$. By Proposition \ref{prop:4+5weakergo}, the Feynman-Kac semigroup is weakly $p$-ergodic and $\phi_0 \in L^p(\mu)$. Let now $g$ be a measurable function such that $|g|\frac{\diff \nu}{\diff \mu} \in L^m(\mu)$. Then
$$ \langle\nu , g e^{\lambda_0 T} \PPP_T^\rho \1 - g \phi_0  \langle \mu, \phi_0 \rangle  \rangle = \int \left[ e^{\lambda_0 T} \PPP_T^\rho \1 (x)- \phi_0(x) \langle \1, \phi_0\rangle_{\mu} \right ]g(x) \frac{\diff \nu}{\diff \mu} (x) \diff \mu(x)     
$$  
which  goes to $0$ as $T\to +\infty$ by Proposition \ref{prop:4+5weakergo}. Since $\frac{\diff \nu}{\diff \mu} \in L^q(\mu) $ with $q>m$, this holds in particular for
$g\in  C_\mathrm{b}(\R^d)$ and we deduce for all such $g$ that  
\begin{equation}\label{eq:convenuTnuinfty}
\langle \nu^T, g\rangle= \frac{\langle \nu, g e^{\lambda_0 T} \PPP_T^\rho\1\rangle}{\langle \nu, e^{\lambda_0 T} \PPP_T^\rho \1\rangle} \xrightarrow[T \to +\infty]{} \frac{\langle \nu, g \phi_0\rangle}{\langle \nu, \phi_0\rangle} = \langle \nu^{\infty}, g\rangle.  
\end{equation}
 Moreover,  by Lemma
 \ref{lemma:exponential_moments}, the function $g=\abs{\cdot}$ is in $L^r(\mu) $ for $r^{-1}=m^{-1}-q^{-1}>0$, and so   $ g\frac{\diff \nu}{\diff \mu} \in L^m(\mu)$ and \eqref{eq:convenuTnuinfty} hold in this case too. We deduce that  $\W_1(\nu^T, \nu^\infty) \xrightarrow[T \to +\infty]{} 0$. 

Since $\frac{\diff \alpha}{\diff \mu} \propto \phi_0 \in L^{p}(\mu)$ and $p > 2 > m$, the preceding argument applies as well to $\alpha$ in the role of $\nu$ (with $p$ in the role of $q$). This ensures that   $\W_1(\alpha^T, \alpha^\infty) \xrightarrow[T \to +\infty]{} 0$.  
 
 Combining the previous convergences in $\W_1$ with the  finite-horizon bound
 $$ \sup_{T\leq T_0} \langle \nu^T, |\cdot |\rangle \leq \frac{\langle \nu, |\cdot |\rangle }{\langle \nu, \PPP^{\rho}_{T_0} \1 \rangle }<\infty $$
 and a similar one for $\alpha$ yields the uniform  boundedness of  $(\alpha^T)_{T>0}$ and $(\nu^T)_{T>0}$ in $  {\cal P}_1(\mathbb{R}^d)$. 
 The first asserted inequality then  follows  thanks to \eqref{eq:proof_contr_aux_1}  and \eqref{eq:proof_contr_aux_2}. 
 
 For the second inequality, in the case $\W_1(\alpha^\infty, \nu^\infty)>0$, we can  take $T_{\nu}>0$  large enough so that $C^* \W_1(\alpha^T, \alpha^\infty)+ C^*\W_1(\nu^T, \nu^\infty) \leq (C'-C^*)  \W_1(\alpha^\infty, \nu^\infty)  $ for all $T\geq T_{\nu}$ which together with \eqref{eq:proof_contr_aux_2} provides the result.  In the case that $\W_1(\alpha^\infty, \nu^\infty)=0$,  for some constant $\theta >0$ we have
$$ \langle \alpha, \phi_0 g\rangle= \langle \theta \nu, \phi_0 g\rangle, \quad
\forall g \in C_\mathrm{b}(\R^d)$$ 
which by strict positivity of $\phi_0$ implies that $\alpha =  \theta \nu$, hence $\alpha = \nu$, and the desired inequality is trivially satisfied.

\end{proof}

\subsection{Proofs of Proposition \ref{prop:q_process} and Theorem \ref{prop:q_process_exp_fast_intro}}
\label{subsec:proofs_qprocess}

\begin{proof}[Proof of Proposition \ref{prop:q_process}]
Let $\nu \ll \mu$ with $\frac{\diff \nu}{\diff \mu} \in L^{p/(p-1)}(\mu)$, and take $S, T \geq 0$ and $F\in C_\mathrm{b}(\bar{\mathbb{D}}_T)$. Then
\begin{align*}
\widetilde{\EE}_\nu^{T+S}[F(\bar X_{[0,T]})] &= \frac{1}{\PP_\nu^{T+S}(N_{T+S}=0)}\EE_\nu^{T+S}[F(\bar X_{[0,T]})\1_{N_{T+S} = 0}]\\
&= \frac{1}{\PP_\nu^{T+S}(N_{T+S}=0)}\EE_\nu^{T+S}[F(\bar X_{[0,T]})\EE_\nu^{T+S}[\1_{N_{T+S} = 0}|\bar{\mathcal{F}}_T]].
\end{align*}
By an argument similar to the one found in the proof of Lemma \ref{lemma:computation_Z}, using the Markov property, we have
\[\EE_\nu^{T+S}[\1_{N_{T+S} = 0}|\bar{\mathcal{F}}_T] = \EE_{X_T}\left[e^{-\int_0^S \rho(X_r)\diff r}\right]\1_{\{N_T=0\}} = f^S(0, X_T) \1_{\{N_T=0\}}.\]
That is,
\begin{align*}
\widetilde{\EE}_\nu^{T+S}[F(X_{[0,T]})] &= \frac{1}{\PP_\nu^{T+S}(N_{T+S}=0)} \EE_\nu^{T+S}\left[F(X_{[0,T]})f^S(0, X_T) \1_{\{N_T=0\}}\right]\\
&= \frac{1}{\int_{\R^d} f^{T+S}(0,x)\diff \nu(x)} \EE_\nu^{T+S}\left[F(X_{[0,T]})f^S(0, X_T) \1_{\{N_T=0\}}\right]\\
&= e^{\lambda_0 T}\frac{\EE_\nu^T\left[F(X_{[0,T]})e^{\lambda_0 S}f^S(0, X_T) \1_{\{N_T=0\}}\right]}{\int_{\R^d} e^{\lambda_0 (T+S)}f^{T+S}(0,x)\diff \nu(x)}.
\end{align*}
By the weak $p$-ergodicity assumption and the integrability of $\nu$, we have
\[\widetilde{\EE}_\nu^{T+S}[F(X_{[0,T]})] \xrightarrow[S \to +\infty]{} e^{\lambda_0 T} \frac{\EE_\nu^T\left[F(X_{[0,T]})\phi_0(X_T) \1_{\{N_T=0\}}\right]}{\int_{\R^d} \phi_0(x) \diff \nu(x)} \eqcolon \langle \QQ^T_\nu, F \rangle.\]
That is,
\[\frac{\diff \QQ^T_\nu}{\diff \PP_\nu^T} = \frac{e^{\lambda_0 T}\phi_0(X_T) \1_{\{N_T=0\}}}{\int_{\R^d} \phi_0(x) \diff \nu(x)}.\]
\end{proof}

Associated with the Q-process is the Q-semigroup, which we now introduce.

\begin{definition}
\label{def:q_semigroup}
We define the \textit{Q-semigroup} $(\QQQ^\rho_t)_{t \geq 0}$ acting on sufficiently integrable functions $g \colon \R^d \to \R$ as follows:
\[\forall t \geq 0, \forall x \in \R^d, \quad \QQQ^\rho_t g (x) \coloneq e^{\lambda_0 t} \frac{\PPP_t^\rho(g\phi_0)(x)}{\phi_0(x)}.\]
\end{definition}

One can see the Q-semigroup as the semigroup of the process conditioned to never be absorbed. More precisely, it is possible to rewrite and adapt the proof of Proposition \ref{prop:q_process} to make appear the Feynman-Kac semigroup for an observable $g \in C_\mathrm{b}(\R^d)$ as follows at a formal level: let $x \in \R^d$, and $s,t \geq 0$. Then
\begin{align*}
\frac{\EE_x\left[g(X_t)e^{-\int_0^{t+s}\rho(X_r)\diff r}\right]}{\EE_x\left[e^{-\int_0^{t+s}\rho(X_r)\diff r}\right]} &= \frac{\PPP_t^\rho(g \PPP_s^\rho\1)(x)}{\PPP_{t}^\rho(\PPP_s^\rho \1)(x)} = \frac{\PPP_t^\rho(g e^{\lambda_0 s} \PPP_s^\rho\1)(x)}{\PPP_{t}^\rho(e^{\lambda_0 s}\PPP_s^\rho \1)(x)}.
\end{align*}
By weak $p$-ergodicity, we have $e^{\lambda_0 s} \PPP_s^\rho\1 \xrightarrow[s \to + \infty]{} \langle \mu, \phi_0 \rangle \phi_0$, so
\[\lim_{s \to + \infty} \frac{\EE_x\left[g(X_t)e^{-\int_0^{t+s}\rho(X_r)\diff r}\right]}{\EE_x\left[e^{-\int_0^{t+s}\rho(X_r)\diff r}\right]} = \frac{\PPP_t^\rho(g\phi_0)(x)}{\PPP_t^\rho\phi_0(x)} = e^{\lambda_0 t} \frac{\PPP_t^\rho(g\phi_0)(x)}{\phi_0(x)} = \QQQ_t^\rho g(x).\]

We remark that the Q-semigroup is a Markov semigroup since for every $t \geq 0$,
\[\QQQ_t^\rho\1 = e^{\lambda_0 t} \frac{\PPP_t^\rho(\phi_0)}{\phi_0} = e^{\lambda_0 t} \frac{e^{-\lambda_0 t}\phi_0}{\phi_0} = \1.\]
Its infinitesimal generator $\LL^\mathrm{Q}$ acts on suitable functions $g \colon \R^d \to \R$ via
\[\LL^\mathrm{Q}g(x) = \lambda_0 g + \frac{1}{\phi_0} \LL^\rho(g\phi_0)\]
and is symmetric with respect to the probability measure $\diff \beta \propto \phi_0^2 \diff \mu \propto \phi_0 \diff \alpha$. Indeed, let $g, h \colon \R^d \to \R$ be smooth test functions. Then, by the symmetry of $\LL^\rho$ with respect to $\mu$,
\begin{align*}
\int_{\R^d} g\LL^\mathrm{Q} h \diff \beta &= \frac{1}{\int_{\R^d}\phi_0^2\diff \mu} \int_{\R^d} g \left(\lambda_0 h + \frac{1}{\phi_0} \LL^\rho(h\phi_0)\right) \phi_0^2 \diff \mu\\
&= \frac{\lambda_0}{\int_{\R^d}\phi_0^2\diff \mu} \int_{\R^d} hg \phi_0^2 \diff \mu + \frac{1}{\int_{\R^d}\phi_0^2\diff \mu} \int_{\R^d} g \phi_0 \LL^\rho(h \phi_0) \diff \mu\\
&= \frac{\lambda_0}{\int_{\R^d}\phi_0^2\diff \mu} \int_{\R^d} hg \phi_0^2 \diff \mu + \frac{1}{\int_{\R^d}\phi_0^2\diff \mu}\int_{\R^d} \LL^\rho(g \phi_0) h \phi_0 \diff \mu\\
&= \frac{1}{\int_{\R^d}\phi_0^2\diff \mu}\int_{\R^d} h \left(\lambda_0 g + \frac1{\phi_0}\LL^\rho(g \phi_0)\right) \phi_0^2 \diff \mu\\
&= \int_{\R^d} h\LL^\mathrm{Q} g \diff \beta.
\end{align*}
Now let us write the operator $\LL^\mathrm{Q}$ in a nicer form: let $g \colon \R^d \to \R$. By the diffusion property of $\LL^X$ and the fact that $\phi_0$ is an eigenfunction for $\LL^\rho$, we have
\begin{align*}
\LL^\mathrm{Q} g &= \lambda_0 g + \frac{1}{\phi_0}\left(g \LL \phi_0 + \phi_0 \LL g + \nabla g \cdot \nabla \phi_0 - \rho g \phi_0\right)\\
&= \LL g + \nabla \ln \phi_0 \cdot \nabla g + \left(\lambda_0 + \frac{\LL\phi_0}{\phi_0} - \rho\right)g\\
&= \LL g + \nabla \ln \phi_0 \cdot \nabla g + \frac{1}{\phi_0}\left(\LL^\rho \phi_0 + \lambda_0 \phi_0\right) g\\
&= \LL g + \nabla \ln \phi_0 \cdot \nabla g\\
&= \frac12 \Delta g - \nabla (U - \ln \phi_0) \cdot \nabla g;
\end{align*}
that is,
\begin{equation}
\label{eq:generator_q_process}
\LL^\mathrm{Q} g = \frac12 \Delta g - \nabla (U - \ln \phi_0) \cdot \nabla g.
\end{equation}

Now we are ready to prove Theorem \ref{prop:q_process_exp_fast_intro}.

\begin{proof}[Proof of Theorem \ref{prop:q_process_exp_fast_intro}]
In light of Theorem \ref{thm:eberle} and part (ii) in Proposition \ref{prop:conforti}, due to our assumptions, the conclusion follows directly.
\end{proof}

\bibliographystyle{alpha}
\bibliography{biblio.bib}

\end{document}